\documentclass[a4paper,10pt]{amsart}
\usepackage[a4paper,tmargin=40mm,bmargin=35mm,hmargin=28mm]{geometry}
\usepackage[T1]{fontenc}
\usepackage{lmodern}
\usepackage{amsmath}
\usepackage{amssymb}
\usepackage{amsthm}
\usepackage{hyperref}
\usepackage[capitalise]{cleveref}
\usepackage{tikz}
\usetikzlibrary{cd}
\usepackage{stmaryrd}
\usepackage[shortlabels]{enumitem}
\makeatletter
  
  \@addtoreset{equation}{section}
\makeatother
\newtheorem{theorem}[equation]{Theorem}
\newtheorem{proposition}[equation]{Proposition}
\newtheorem{lemma}[equation]{Lemma}
\newtheorem{corollary}[equation]{Corollary}
\newtheorem*{theorem*}{Theorem}
\theoremstyle{definition}

\newtheorem{remark}[equation]{Remark}

\DeclareMathOperator{\Spec}{Spec}
\DeclareMathOperator{\Max}{Max}

\DeclareMathOperator{\rad}{rad}
\DeclareMathOperator{\depth}{depth}

\DeclareMathOperator{\gld}{gld}

\DeclareMathOperator{\Zg}{Zg}
\let\mod\relax
\DeclareMathOperator{\mod}{\mathsf{mod}}
\DeclareMathOperator{\Mod}{\mathsf{Mod}}

\DeclareMathOperator{\Flat}{\mathsf{Flat}}
\DeclareMathOperator{\Proj}{\mathsf{Proj}}
\DeclareMathOperator{\GProj}{\mathsf{GProj}}
\DeclareMathOperator{\Gproj}{\mathsf{Gproj}}
\DeclareMathOperator{\sGProj}{\underline{\mathsf{GProj}}}
\DeclareMathOperator{\sGproj}{\underline{\mathsf{Gproj}}}
\DeclareMathOperator{\proj}{\mathsf{proj}}
\DeclareMathOperator{\CM}{\mathsf{CM}}
\DeclareMathOperator{\sCM}{\underline{\mathsf{CM}}}

\DeclareMathOperator{\add}{\mathsf{add}}
\DeclareMathOperator{\Ab}{Ab}
\DeclareMathOperator{\fp}{fp}
\DeclareMathOperator{\id}{\mathrm{id}}
\DeclareMathOperator{\Hom}{Hom}
\DeclareMathOperator{\End}{End}
\DeclareMathOperator{\Ext}{\mathrm{Ext}}

\DeclareMathOperator{\RGamma}{\mathrm{R}\Gamma}

\DeclareMathOperator{\Ker}{\mathrm{Ker}}
\DeclareMathOperator{\Coker}{\mathrm{Coker}}

\DeclareMathOperator{\PH}{P}

\DeclareMathOperator{\Z}{\mathrm{Z}}
\newcommand{\pp}{\mathfrak{p}}
\newcommand{\qq}{\mathfrak{q}}
\newcommand{\mm}{\mathfrak{m}}
\newcommand{\nn}{\mathfrak{n}}
\newcommand{\fa}{\mathfrak{a}}
\newcommand{\fb}{\mathfrak{b}}
\newcommand{\cA}{\mathcal{A}}
\newcommand{\cB}{\mathcal{B}}

\newcommand{\cC}{\mathcal{C}}
\newcommand{\cT}{\mathcal{T}}

\newcommand{\ZZ}{\mathbb{Z}}

\renewcommand{\H}{\mathrm{H}}
\newcommand{\K}{\mathsf{K}}
\newcommand{\D}{\mathsf{D}}

\newcommand{\ac}{\mathsf{ac}}

\newcommand{\tac}{\mathsf{tac}}
\definecolor{mygreen}{rgb}{0.1,0.6,0.3}

\newcommand{\isoto}{\xrightarrow{\smash{\raisebox{-0.25em}{$\sim$}}}}
\newcommand{\hr}{\hookrightarrow}

\newcommand{\op}{\mathrm{op}}
\newcommand{\cp}{\mathrm{c}}

\crefname{enumi}{}{}
\Crefname{enumi}{}{}
\creflabelformat{enumi}{\normalfont #2#1#3}
\crefrangeformat{enumi}{\normalfont #3#1#4--#5#2#6}
\crefname{enumii}{}{}
\Crefname{enumii}{}{}
\creflabelformat{enumii}{\normalfont #2#1#3}
\crefrangeformat{enumii}{\normalfont #3#1#4--#5#2#6}
\crefname{enumiii}{}{}
\Crefname{enumiii}{}{}
\creflabelformat{enumiii}{\normalfont #2#1#3}
\crefrangeformat{enumiii}{\normalfont #3#1#4--#5#2#6}

\makeatletter
\renewcommand{\p@enumii}{}
\renewcommand{\p@enumiii}{}
\makeatother

\numberwithin{equation}{section}
\crefname{equation}{}{}
\Crefname{equation}{}{}
\creflabelformat{equation}{(#2#1#3)}

\makeatletter
\@namedef{subjclassname@2020}{\textup{2020} Mathematics Subject Classification}
\makeatother
\makeatletter
\let\@wraptoccontribs\wraptoccontribs
\makeatother

\title[Indecomposable pure-injective objects in stable categories]{Indecomposable pure-injective objects in stable categories of Gorenstein-projective modules over Gorenstein orders}

\thanks{Tsutomu Nakamura was supported by PRIN-2017  ``Categories, Algebras: Ring-Theoretical and Homological Approaches (CARTHA)'' and Grant-in-Aid for JSPS Fellows JP20J01865.}

\author{Tsutomu Nakamura}
\address[T. Nakamura]{Department of Mathematics, Faculty of Education, Mie University, 1577 Kurimamachiya-cho, Tsu city, Mie, 514-8507, Japan}
\email{nakamura@edu.mie-u.ac.jp}

\contrib[with an appendix by]{Rosanna Laking}
\address[R. Laking]{Dipartimento di Informatica - Settore di Matematica, Universit\`a degli Studi di Verona, Strada le Grazie 15 - Ca' Vignal, I-37134 Verona, Italy}
\email{rossana.laking@univr.it}

\begin{document}

\begin{abstract}
We give a result of Auslander--Ringel--Tachikawa type for Gorenstein-projective modules over a complete Gorenstein order.
In particular, we prove that a complete Gorenstein order is of finite Cohen--Macaulay representation type if and only if every indecomposable pure-injective object in the stable category of Gorenstein-projective modules is compact.
\end{abstract}

\maketitle
\tableofcontents

%%%%%%%%%%%%%%%%%%%%%%%%%%
%%%%%%%%%%%%%%%%%%%%%%%%%%
\section{Introduction}
Ringel and Tachikawa \cite{RT74} proved that 
if an Artin algebra $A$ is of finite representation type, then every $A$-module is a direct sum of finitely generated $A$-modules. The converse implication also holds due to Auslander \cite{Aus76}, where he further proved that an Artin algebra $A$ is of finite representation type if and only if every indecomposable $A$-module is finitely generated.
This equivalence, which characterizes infinite representation type by the existence of infinitely generated indecomposable modules, makes us pay attention to the class of large (i.e., possibly infinitely generated) indecomposable modules.
However, this class is in general too huge to talk about representation theory, because it can be a proper class even for a finite-dimensional algebra (\cite[Theorem 3.1]{Sim05}).
This problem can be avoided if we focus on the indecomposable \emph{pure-injective} modules.

Pure-injectivity originally appeared as \emph{algebraic compactness}, which is a property of modules over a ring to have solutions for systems of linear equations; see \cite[Chapter VII]{Fuc70} and \cite[Chapter 7]{JL89} for example.
Ziegler discovered in his model theoretic work \cite{Zie84}  that the isomorphism classes of indecomposable algebraically compact modules form a small set endowed with a natural topology. This topological space is called the \emph{Ziegler spectrum} of the ring.

It is well known that every finitely generated modules over an Artin algebra $A$ is algebraically compact (i.e., pure-injective), and hence the Ziegler spectrum contains, up to isomorphism, all indecomposable finitely generated modules.
In this sense, the Ziegler spectrum is a reasonable place where we can discuss representation theory of algebras including large indecomposable modules. 
Furthermore, giving another proof for the result of Auslander--Ringel--Tachikawa, model theory of modules provides a refined statement: 
\begin{theorem}\label{Artin}
Let $A$ be an Artin algebra.
The following conditions are equivalent.
\begin{enumerate}[label=(\arabic*), font=\normalfont]
\item \label{Artin-finite} $A$ is of finite representation type.
\item \label{Artin-ART} Every right $A$-module is a direct sum of finitely generated right $A$-modules.
\item \label{Artin-pure-inj} Every indecomposable pure-injective right $A$-module is finitely generated.
\end{enumerate}
\end{theorem}
If ``pure-injective'' is removed from \cref{Artin-finite}, this theorem is precisely due to Auslander--Ringel--Tachikawa.
One can find a proof of \cref{Artin} in \cite[\S 5.3.4]{Pre09}, where topological data of the Ziegler spectrum plays an important role.

Among indecomposable large pure-injective modules, \emph{generic} modules have been well studied. They are infinitely generated indecomposable modules of finite endolength, and this notion was introduced by Crawley-Boevey \cite{CB91}. His work characterizes when finite-dimensional algebra over algebraically closed fields have tame representation type in terms of generic modules. 
Furthermore, generic modules are closely related to tubes in Auslander--Reiten quivers; see Krause \cite{Kra98} and Crawley-Boevey \cite{CB98}.

Ringel \cite{Rin00,Rin11} introduced the notion of \emph{Auslander--Reiten quilts} for $1$-domestic special biserial algebras. 
This compactifies and sews connected components of Auslander--Reiten quivers of such algebras by adding infinitely generated indecomposable pure-injective modules. In consequence, we obtain topological pictures, which yield a better understanding of relationship among the connected components.
Puninski \cite{Pun18} explained that a similar procedure is possible for a plane curve singularity of type $A_\infty$, classifying all indecomposable pure-injective ``infinitely generated Cohen--Macaulay modules'' up to isomorphism, where he considered a (possibly infinitely generated) module to be Cohen--Macaulay if it does not admit any nonzero morphism from the residue field. He and Los \cite{LP19} further classified such indecomposable pure-injective infinitely generated modules for a plane curve singularity of type $D_\infty$. Sadly, Puninski passed away before that paper was published.

For the singularities considered in \cite{Pun18} and \cite{LP19},
the Auslander--Reiten quivers of finitely generated Cohen--Macaulay (CM) modules 
had been well known and they are of countable CM representation type. This fact might have been sufficient for Puninski to start his attempt, while there was no result ensuring the existence of indecomposable infinitely generated CM modules.
We want to dispel this uncertain situation to continue the interesting research of infinitely generated CM representations, under a firm foundation.
Towards this purpose, we prove the following result, which is a Gorenstein-projective analogue to \cref{Artin}. 

\begin{theorem}\label{main-theorem}
Let $R$ be a complete CM local ring and $A$ a Gorenstein $R$-order.
The following conditions are equivalent:
\begin{enumerate}[label=(\arabic*), font=\normalfont]
\item \label{fCM} $A$ is of finite CM representation type.
\item \label{GP} Every Gorenstein-projective right $A$-module is a direct sum of finitely generated Gorenstein-projective right $A$-modules.
\item \label{cCM} Every indecomposable pure-injective object in the stable category $\sGProj A$ of Gorenstein-projective right $A$-modules is compact.   
\end{enumerate}
\end{theorem}

Here, an $R$-order means a module-finite $R$-algebra which is a nonzero maximal CM $R$-module. 
The Gorenstein $R$-order $A$ is an $R$-order such that $\Hom_R(A,\omega_R)$ is projective as a right $A$-module, where $\omega_R$ stands for a canonical module of $R$. 

Let us first compare the above theorem with existing results.
The equivalence \cref{fCM}$\Leftrightarrow$\cref{GP} in \cref{main-theorem} is a partial generalization of \cite[Theorem 4.20]{Bel11} due to Beligiannis, whose result treats a complete Gorenstein local ring. 
It is also possible to deduce the equivalence \cref{fCM}$\Leftrightarrow$\cref{GP} from recent work by Psaroudakis and Rump \cite[Corollary 5.14]{PR22} where they extended Chen's result \cite[Main theorem]{Che08} to a wide class of Iwanaga--Gorenstein rings satisfying some conditions. The result of  Chen shows that an Iwanaga--Gorenstein Artin algebra has up to isomorphism only finitely many indecomposable Gorenstein-projective modules if and only if the condition \cref{GP} holds.
The Gorenstein $R$-order $A$ in \cref{main-theorem} is Iwanaga--Gorenstein, but in general an $R$-order being Iwanaga--Gorenstein may not be Gorenstein in our sense. Considering Gorenstein $R$-orders is important for our purpose so that we can identify the finitely generated Gorenstein-projective $A$-modules with the maximal CM $A$-modules; see \cref{CM-Gproj}.
If $R$ in \cref{main-theorem} is artinian, the Gorenstein $R$-order $A$ is nothing but a quasi-Frobenius $R$-algebra, so \cref{main-theorem} follows from \cref{Artin} and Krause's \cite[Proposition 1.16]{Kra00}. 
In the case where $R$ is not artinian, \cref{main-theorem}\cref{cCM} is actually a new characterization for $A$ to have finite CM representation type (even if $R=A$).

We next explain an outline of the proof of \cref{main-theorem}.
Our approach for the implication \cref{fCM}$\Rightarrow$\cref{GP} of \cref{main-theorem} is similar to Beligiannis' idea, but we adopt a slightly different way, which would be simpler and closer to a classical argument; cf. \cite[Theorems 4.5.4 and 4.5.7 and Proposition 4.5.6]{Pre09}.
The implication \cref{GP}$\Rightarrow$\cref{cCM} is almost trivial once we know what the compact objects in $\sGProj A$; see \cref{comp-Gor-case}\cref{equiv-comp}.
A difficult part is to prove the implication \cref{cCM}$\Rightarrow$\cref{fCM}. In general, the Ziegler spectrum of a ring is quasi-compact as a topological space (\cite[Corollary 5.123]{Pre09}), and this fact is used (in the proof of \cref{Artin} by \cite[Theorem 5.3.40]{Pre09}) to find a infinitely generated indecomposable pure-injective module over an Artin algebra of infinite representation type. However, there is no reason that the Ziegler spectrum of the stable category $\sGProj A$ is quasi-compact.
We handle this problem by using \cref{Existence,Fundam,Generator}. 
It is also non-trivial that the condition \cref{cCM} makes $A$ have at most an isolated singularity. This is done in \cref{non-isolate}.

Lastly we make a remark which relates the present work to balanced big CM modules.

\begin{remark}\label{BCM-remark}
Recall that a module $M$ over a commutative noetherian local ring $(R,\mm)$ is called \emph{balanced big CM} if every system of parameters of $R$ is an $M$-regular sequence (\cite[\S 8.5]{BH98}).
Let $R$ and $A$ be as in \cref{main-theorem} and define a balanced big CM right $A$-module as a right $A$-module which is balanced big CM over $R$. The author of the present paper will show in his subsequent work \cite{Nak21a} that each of \cref{fCM,GP,cCM} in \cref{main-theorem} is further equivalent to the following condition:
\begin{enumerate}[label=($*$)]
\item \label{star}\emph{Every indecomposable pure-injective balanced big CM right $A$-module is finitely generated.}
\end{enumerate}
In other words, $A$ has infinite CM representation type if and only if there exists an indecomposable pure-injective balanced big CM $A$-module that is infinitely generated.
\cref{main-theorem} is one of crucial steps to give this characterization. Moreover, we can regarded \cref{star} as a CM analogue to \cref{Artin}\cref{Artin-pure-inj}.

The present paper will not focus on balanced big CM modules, but we refer to some recent results about balanced big CM modules over orders.
For a complete CM local ring $(R,\mm)$, Bahlekeh, Fotouhi, and Salarian \cite[Theorems 6.7]{BFS19} showed that $R$ is of finite CM representation type  if and only if every balanced big CM module having an \emph{$\mm$-primary cohomological annihilator} (see \cite[p.~1675]{BFS19}) decomposes into a direct sums of finitely generated ones (i.e., small CM modules). Using this fact, they also recovered the equivalence \cref{fCM}$\Leftrightarrow$\cref{GP} of \cref{main-theorem} for a complete Gorenstein local ring due to Beligiannis; see 
\cite[Theorems 6.8]{BFS19}.
On the other hand, Psaroudakis and Rump proved \cite[Theorem 5.18]{PR22} that  an $R$-order $A$ for a complete regular local ring $R$ is of finite CM representation type if and only if  every \emph{accessible} balanced big CM $A$-module (see \cite[p. 20 and Definition 5.6]{PR22}) decomposes into a direct sum of finitely generated ones.
\end{remark}

%%%%%%%%%%%%%%%%%%%%%%%%%%
%%%%%%%%%%%%%%%%%%%%%%%%%%
\section{Preliminaries}\label{preliminaries}
In this section, we collect various facts to prove the main theorem (\cref{main-theorem}).

Throughout of this paper, a ring $A$ will always mean an associative ring with identity, and an $A$-module will mean a right $A$-module unless otherwise specified. Left $A$-modules are interpreted as right modules over the opposite ring $A^\op$.

%%%%%%%%%%%%%%%%%%%%
%%%%%%%%%%%%%%%%%%%%
\subsection{Gorenstein-projective modules}\label{subsect-Gproj}
Let $A$ be a ring.
We denote by $\Proj A$ (resp.~$\proj A$) the category of projective (resp.~finitely generated projective) $A$-modules.
A complex $X=(\cdots \to X^{i}\to X^{i+1}\to \cdots)$ of projective $A$-modules is said to be \emph{totally acyclic} if it is acyclic (i.e., $H^iX=0$ for all $i\in \ZZ$) and the complex $\Hom_A(X,P)$ is acyclic for any projective $A$-module $P$.
An $A$-module $M$ is said to be \emph{Gorenstein-projective} if there exists a totally acyclic complex $X$ of projective $A$-modules such that  $M=\Z^0X:=\Ker(X^0 \to X^{1})$. If this is the case, $X$ is called a \emph{complete resolution} of $M$.

Assume that $A$ is a (two-sided) coherent ring and $M$ is a Gorenstein-projective $A$-module that is finitely presented.  Then $M$ admits a complete resolution by finitely generated projective $A$-modules (see \cite[Proposition 10.2.6]{EJ11} and its preceding paragraphs).
Denote by $\GProj A$ (resp.~$\Gproj A$) the category of Gorenstein-projective (resp.~finitely generated Gorenstein-projective)  $A$-modules.
It is well known that $\GProj A$ (resp.~$\Gproj A$) is a Frobenius category whose projective-injective objects are the projective (resp.~finitely generated projective) $A$-modules; see, e.g., \cite[Proposition 3.1 and p. 207]{Che11} or \cite[Corollary 11.2.6]{Per16}.
Hence the stable category $\sGProj A$ (resp.~$\sGproj A$)  of $\GProj A$ (resp.~$\Gproj A$) are triangulated (\cite[Chapter I, Theorem 2.6]{Hap88}). Note that $\sGProj A$ (resp.~$\sGproj A$) is the stable category associated with $\GProj A$ (resp.~$\Gproj A$) in the sense of \cite[Chapter I, \S 2.2]{Hap88}.

Since every morphism from a finitely presented $A$-module to a projective $A$-module factors through a finitely generated projective $A$-module, there is a fully faithful functor $\sGproj A\to \sGProj A$ given by $M\mapsto M$.
Each (distinguished) triangle in $\sGProj A$ (resp.~$\sGproj A$) essentially arise from a short exact sequence in $\GProj A$ (resp.~$\Gproj A$) (\cite[\SS 2.5 and 2.7]{Hap88}), and so the canonical functor $\sGproj A\to \sGProj A$ is triangulated (\cite[\S 2.8]{Hap88}). Then we can regard $\sGproj A$ as a triangulated full subcategory of $\sGProj A$.

Denote by $\K_{\tac}(\Proj A)$ (resp.~$\K_{\tac}(\proj A)$) the homotopy category of totally acyclic complexes of projective (resp.~finitely generated projective) $A$-modules.
It is well known that taking complete resolutions by projective $A$-modules yields a triangulated equivalence $\sGProj A\isoto \K_{\tac}(\Proj A)$ whose quasi-inverse is induced by $\Z^0$ (cf. \cite[Theorem 4.4.1(1)]{Buc86} and \cite[Proposition 7.2]{Kra05}). This triangulated equivalence restricts to a triangulated equivalence $\sGproj A\isoto \K_{\tac}(\proj A)$ as $A$ is coherent.

Denote by $\K_{\ac}(\Proj A)$ (resp.~$\K_{\ac}(\proj A)$) the homotopy category of acyclic complexes of projective  (resp.~finitely generated projective) $A$-modules.
Assume that $A$ is \emph{Iwanaga--Gorenstein}, i.e., $A$ is a (two-sided) noetherian ring such that it has finite injective dimension as a left and ring $A$-module. 
This assumption implies that every projective $A$-module has finite injective dimension (\cite[Theorem 9.1.10]{EJ11}), so we have
$\K_{\tac}(\Proj A)=\K_{\ac}(\Proj A)$ and $\K_{\tac}(\proj A)=\K_{\ac}(\proj A)$ by a standard argument (cf. \cite[Example 7.15]{Kra05}, \cite[Corollary 4.28]{MS11}, and \cite[Lemma 4.3]{Che11}).
Consequently, using the canonical functors $\sGproj A\hr \sGProj A$ and $\K_{\ac}(\proj A)\hr \K_{\ac}(\Proj A)$, we obtain the following quasi-commutative diagram:
\begin{equation}\label{diagram}
\begin{tikzcd}
\K_{\ac}(\Proj A)\ar[r, "\Z^0"',"\sim"]&\sGProj A\\
\K_{\ac}(\proj A)\ar[r, "\Z^0"',"\sim"]\ar[u,hookrightarrow]&\sGproj A\ar[u,hookrightarrow]
\end{tikzcd}
\end{equation}

\begin{remark}\label{cp}
By J{\o}rgensen \cite[Theorem 2.4]{Jor05}, the homotopy category $\K(\Proj A)$ is compactly generated if $A$ is a coherent ring and  every flat right $A$-module has finite projective dimension.
More generally, Neeman \cite[Theorem 1.1]{Nee08} shows that $\K(\Proj A)$ is compactly generated whenever $A$ is a left coherent ring.
It is well known to experts that the compact generation of $\K(\Proj A)$ implies that $\K_\ac(\Proj A)$ is compactly generated; see  \cite[Corollary 5.17]{Mur07}.
\end{remark}

When $A$ is a Iwanaga--Gorenstein ring, every flat right $A$-module has finite projective dimension by \cite[Proposition~9.1.2]{EJ11}.
Hence the compact generation of $\K(\Proj A)$ can be deduced from \cite[Theorem 2.4]{Jor05} (see \cite[Lemma~4.2]{Che11}). 
Moreover, in this case, $\K_\ac(\Proj A)$ is triangulated equivalent to $\sGProj A$ as explained above. Therefore $\sGProj A$ is compactly generated. This fact is shown in \cite[Theorem 4.1]{Che11}; more precisely, it shows the following:

\begin{theorem}\label{compact}
Let $A$ be an Iwanaga--Gorenstein ring.
The triangulated category $\sGProj A$ is compactly generated, and each compact object in $\sGProj A$ is a direct summand of some object in the essential image of the canonical functor $\sGproj A\hookrightarrow \sGProj A$.
\end{theorem}

For completing the proof of the main theorem, the reader may assume that all rings in this paper are Iwanaga--Gorenstein, although we will not do this so that we can see what is essential in each result.

%%%%%%%%%%%%%%%%%%%%
%%%%%%%%%%%%%%%%%%%%
\subsection{Pure-injectivity and functors commuting with small direct products}
\label{PI-R}
A morphism $f:M\to N$ of right modules over a ring $A$ is said to be a \emph{pure monomorphism} if the induced morphism $L\otimes_A f$ is a monomorphism for every left $A$-module $L$, or equivalently, if $f$ induces a short exact sequence
$0\to \Hom_A(F,M)\to \Hom_A(F,N)\to \Hom_A(F,\Coker f)\to 0$ for every finitely presented $A$-module $F$ (see \cite[Theorem 6.4]{JL89}).
An $A$-module $P$ is said to be \emph{pure-injective} if the functor $\Hom_A(-,P)$ sends every pure monomorphism $M\to N$ of $A$-modules to an epimorphism $\Hom_A(N,P)\to \Hom_A(M,P)$.
Pure-injectivity of objects in a compactly generated triangulated category is defined analogously (\cite{Bel00}, \cite{Kra00}); see \cref{PointsZiegler}.

Let $\cA$ be an additive category, and assume that $\cA$ admits small direct sums and small direct products, i.e., for any family $\{X_i\}_{i\in I}$ of objects $X_i$ in $\cA$ with $I$ a small set, the direct sum $\bigoplus_{i\in I}X_i$ and the direct product $\prod_{i\in I}X_i$ exist in $\cA$.
Given an object $X\in \cA$ and a small set $I$, we write $X^{(I)}:=\bigoplus_{i\in I} X_i$ and $X^{I}:=\prod_{i\in I} X_i$, setting $X_i:=X$ for each $i\in I$.
The \emph{summation morphism} $X^{(I)}\to X$ is defined to be the morphism induced by the identity $X_i:=X\to X$ for each $i\in I$.

A module $M$ over a ring is pure-injective if and only if the summation morphism $M^{(I)}\to M$ factors through the canonical morphism $M^{(I)}\to M^{I}$ (see \cite[Theorem 7.1]{JL89}). The same characterization of pure-injectivity is available in a compactly generated triangulated category $\cT$; see \cref{injectives}.

Let $\cA$ be an additive category with small direct sums and small direct products, i.e., $\cA$ is an additive category that admits small direct sums and small direct products.
Following \cite[Definition 5.1(1)]{SS20a}, we say that an object $X\in \cA$ is \emph{pure-injective} if, for any small set $I$, the summation morphism $X^{(I)}\to X$ factors through the canonical morphism $X^{(I)}\to X^I$.

\begin{proposition}\label{pure-prod}
Let $\cA$ and $\cB$ be additive categories with small direct sums and small direct products. 
Let $G:\cA\to \cB$ be an additive functor that commutes with small direct products.
If $X$ is a pure-injective object in $\cA$, then $GX$ is pure-injective in $\cB$. 
\end{proposition}

\begin{proof}
Let $X$ be a pure-injective object in $\cA$.
Let $I$ be a small set,  $c: X^{(I)}\to X^I$ the canonical morphism, and $s: X^{(I)}\to X$ the summation morphism.
Since $X$ is pure-injective, there exists a morphism $t:X^{I}\to X$ making the following diagram commutative:
\[
\begin{tikzcd}
X^{(I)}\ar[rr,"c"]\arrow[rd,"s"'] &&X^I\ar[ld,"t"]\\
&X
\end{tikzcd}
\]
Since $G$ commutes with small direct products, we have $G(X^I)=G(X)^I$.
For each $j\in I$, let $e_j:X_j\to \bigoplus_{i\in I}X_i=X^{(I)}$ be the canonical injection, where $X_i:=X$. 
Let $e: G(X)^{(I)}\to G(X^{(I)})$ be the morphism induced by $G(e_i): GX_i\to G(X^{(I)})$ for each $i\in I$.
Then we have the following commutative diagram:
\[
\begin{tikzcd}
 (GX)^{(I)}\ar[r,"e"] &G (X^{(I)})\ar[rr," G(c)"]\arrow[rd,"G(s)"'] &&G(X^I)=(GX)^I\ar[ld,"G(t)"]\\
&& GX
\end{tikzcd}
\]
By construction, the composition $G(c)\cdot e$ is the canonical morphism
$G(X)^{(I)}\to G(X)^{I}$, and the composition $G(s)\cdot e$ is the summation morphism $G(X)^{(I)}\to GX$.
Then the above diagram shows that $GX$ is pure-injective in $\cB$.
\end{proof}

\begin{corollary}\label{pure-right}
Let $\cA$ and $\cB$ be additive categories with small direct sums and small direct products. 
Let $G:\cA\to \cB$ be a right adjoint to some additive functor $\cB\to \cA$. If $X$ is a pure-injective object in $\cA$, then $GX$ is pure-injective in $\cB$. 
\end{corollary}

\begin{proof}
Since $G$ is a right adjoint, it commutes with small direct products, so the corollary follows from \cref{pure-prod}.
\end{proof}

We will use this corollary in the proof of \cref{non-isolate}.

%%%%%%%%%%%%%%%%%%%%
%%%%%%%%%%%%%%%%%%%%
\subsection{Injective modules and injective dimension over Noether algebras}\label{subsect-Noeth}
Let $R$ be a commutative noetherian ring. 
Let $A$ be a \emph{Noether $R$-algebra} $A$, that is, a ring $A$ together with a ring homomorphism $\varphi: R\to A$ such that  
the image $\varphi(R)$ is contained in the center of $A$ and $A$ is finitely generated as an $R$-module.
We call $\varphi$ the \emph{structure map} of $A$. We denote by $\Mod A$ (resp.~$\mod A$) the category of right $A$-modules (resp.~finitely generated right $A$-modules).
Given a prime ideal $\pp$ of $R$, we say that $M\in \Mod A$ is $\pp$\emph{-local} if the canonical $A$-homomorphism $M\to M_\pp:=M\otimes_R R_\pp$ is bijective.
Moreover, given an ideal $\fa$ of $R$, we say that $M\in \Mod A$ is $\fa$\emph{-torsion} if the canonical $A$-homomorphism $\Gamma_\fa M:=\varinjlim_{n\geq 1}\Hom_R(R/{\fa},M)\to M$ is bijective.

We denote by $\Spec A$ the set of (two-sided) prime ideals of $A$ .
For each $P\in \Spec A$, $\pp:=\varphi^{-1}(P)$ is a prime ideal of $R$, $P_\pp$ is a maximal ideal of the Noether $R_\pp$-algebra $A_\pp$, and $A_\pp/P_\pp$ decomposes into a finite direct sum of copies of a simple (right) $A_\pp$-module $S_A(P)$ (see \cref{ss-dec}).
We denote by $I_A(P)$ the injective envelope of $S_A(P)$ in $\Mod A$. 
Then $I_A(P)$ is indecomposable as an $A$-module, and it is $\pp$-local and $\pp$-torsion.
Moreover, there is a canonical bijection between $\Spec A$ and the set of isomorphism classes of indecomposable injective $A$-modules given by $P\mapsto I_A(P)$. See \cite[\S\S 2.3--2.4]{KN22} for more details.

Since $A$ is noetherian, every injective $A$-module $I$ into a direct sum of indecomposable injective $A$-modules. By the above-mentioned bijection, we have 
\begin{equation*}
 I\cong \bigoplus_{P\in\Spec A}I_{A}(P)^{(C_{P})}
	\end{equation*}
for some family of sets $\{C_{P}\}_{P\in\Spec A}$.
Then we have
\begin{equation}\label{indec-gamma-0}
		\Gamma_{\mathfrak{a}} I\cong \bigoplus_{\substack{P\in \Spec A\\ \mathfrak{a}\subseteq \varphi^{-1}(P)}}I_{A}(P)^{(C_{P})}
	\end{equation}
for an ideal $\fa$ of $R$; see \cite[Remark 3.4(1)]{KN22}.
If $R$ is local and $\mm$ is the maximal ideal of $R$, then
\begin{equation}\label{indec-gamma}
		\Gamma_\mm I\cong \bigoplus_{P\in \Max A}I_{A}(P)^{(C_{P})},
	\end{equation}
where $\Max A$ denotes the set of maximal ideals of $\Spec A$; see \cite[Lemma 2.12]{KN22}. Note that 
$\Max A$ is a finite set when $R$ is local; see \cite[Propositions  2.14(2) and 2.15]{KN22}.

We also have an isomorphism
\begin{equation*}\label{indec-localize-0}
		I\otimes_R S^{-1}R\cong \bigoplus_{\substack{P\in \Spec A\\ \varphi^{-1}(P)\cap S=\emptyset}}I_{A}(P)^{(C_{P})}
	\end{equation*}
for a multiplicatively closed subset $S$ of $R$; see \cite[Remark 3.4(2)]{KN22}. Hence if $S$ is the multiplicatively closed subset generated by an element $x\in R$, then
\begin{equation}\label{indec-localize}
		I_x\cong \bigoplus_{\substack{P\in \Spec A\\ x\notin \varphi^{-1}(P)}}I_{A}(P)^{(C_{P})},
	\end{equation}
where $M_x:=M \otimes_R S^{-1}R$ for an $A$-module $M$.

\begin{remark}\label{injective}
Let $(R,\mm)$ be a commutative noetherian local ring and $A$ a Noether $R$-algebra. Denote by $\rad A$ the Jacobson radical of $A$.
Let $M\in \Mod A$ and let $f: M\hr E_A(M)$ be an injective envelope in $\Mod A$.
We observe that the induced map \[\Hom_A(A/\rad A,f):\Hom_A(A/\rad A, M)\to\Hom_A(A/\rad A, E_A(M))\]
is an isomorphism. 
To see this, recall that $A/\rad A$ decomposes as
\begin{equation}\label{ss-dec}
A/\rad A\cong \bigoplus_{P\in \Max A} S_A(P)^{n_P}
\end{equation}
for some integers $n_P\geq 0$; see \cite[Proposition 2.19]{KN22}.
Thus it suffices to show that the inclusion 
$\Hom_A(S_A(P),f): \Hom_A(S_A(P), M)\hr \Hom_A(S_A(P), E_A(M))$ is surjective for each $P\in \Max A$.
However this is obvious because $f$ is an essential extension and $S_A(P)$ is a simple $A$-module.

Let $M$ be an $A$-module and let $E=(0\to E^0 \to E^1\to \cdots )$ be a minimal injective resolution of $M$. By the above observation, we see that the complex $\Hom_A(A/\rad A, E)$ has zero differential.
Moreover, if we write $E^i=\bigoplus_{P\in \Spec A} I_A(P)^{(C_P^i)}$ for each $i\geq 0$, then 
\begin{align*}
\Hom_A(A/\rad A, E^i)
&\cong \Hom_A(A/\rad A, \Gamma_\mm E^i)\\
&\cong \Hom_A(A/\rad A, \bigoplus_{P\in \Max A} I_A(P)^{(C_P^i)})\\
&\cong \bigoplus_{P\in \Max A} \Hom_A(A/\rad A, I_A(P)^{(C_P^i)})\\
&\cong \bigoplus_{P\in \Max A} S_A(P)^{(C_P^i)}
\end{align*}
in $\Mod A$, where the first isomorphism holds since $A/\rad A$ is $\mm$-torsion (see \cite[Lemma 2.12 and Remark 2.23]{KN22}), the second follows from \cref{indec-gamma}, and the last follows from \cite[Proposition 4.8(2)]{KN22}.
Therefore we see that $\Hom_A(A/\rad A, E^i)=0$ if and only if $\Gamma_\mm E^i=0$.

The isomorphism $\Hom_A(A/\rad A, E^i)\cong \bigoplus_{P\in \Max A} S_A(P)^{(C_P^i)}$ can be also deduced from \cite[Lemma 5.1(2)]{GN02}; its proof implicitly uses the fact that $\Hom_A(A/\rad A, E)$ has zero differential. 
\end{remark}

Given an $A$-module $M$, we denote by $\id_A M$ the injective dimension of $M$.

\begin{lemma}\label{id-test}
Let $R$ be a commutative noetherian local ring and $A$ a Noether $R$-algebra.
Let $M$ be a nonzero finitely generated $A$-module. 
Then $\id_A M=\sup \{i \mid \Ext^i_A(A/\rad A, M)\neq 0\}$.
\end{lemma}

\begin{proof}
See \cite[Corollary 3.5(3)]{GN02}.
\end{proof}

%%%%%%%%%%%%%%%%%%%%
%%%%%%%%%%%%%%%%%%%%
\subsection{Local cohomology of modules over Noether algebras}\label{loc-coh}

Let $R$ be a commutative noetherian ring $R$ and 
denote by $\D(R)$ the unbounded derived category of $R$-modules.
Let $\fa$ be an ideal of $R$ and consider the right derived functor $\RGamma_\fa: \D(R)\to \D(R)$ of the $\fa$-torsion functor $\Gamma_\fa:\Mod R\to \Mod R$.
The functor $\H^i\RGamma_\fa: \D(R)\to \Mod R$ is denoted by $\H^i_\fa$ and called the \emph{$i$th local cohomology functor with respect to $\fa$}.

Let $\boldsymbol{x}=x_1,\ldots,x_n$ be a system of generators of $\fa$. 
For each $x_i$, consider the complex $(R \to R_{x_i})$ concentrated in degrees $0$ and $1$, where the morphism $R \to R_{x_i}$ is the localization map. 
The \emph{(extended) \v{C}ech complex with respect to} $\boldsymbol{x}$ is defined as 
$\check{C}(\boldsymbol{x}):=\bigotimes_{i=1}^n (R \to R_{x_i})$. 
For every $X\in \D(R)$, there is a natural isomorphism 
\begin{align}\label{Cech-comp}
\RGamma_{\fa}X\cong \check{C}(\boldsymbol{x}) \otimes_R X
\end{align}
in $\D(R)$; see \cite[Proposition 3.1.2]{Lip02}. 
In particular, given an $R$-module $M$, we have $\H^i_\fa M=0$ for every $i>n$ (\cite[Corollary 3.1.4]{Lip02}). 
If $\fa$ and $\fb$ are ideals of $R$ defining the same Zariski closed subset, we have $\Gamma_{\fa}=\Gamma_\fb$ as functors $\Mod R\to \Mod R$; see \cite[Exercise 1.1.3]{BS13} and \cite[p.~3]{Mat89}.

\begin{remark}\label{Groth-vanish}
Let $(R,\mm)$ be a commutative noetherian local ring and $d:=\dim R$, where $\dim R$ stands for the Krull dimension of $R$.  Let $\boldsymbol{x}=x_1,\ldots,x_d$ be a system of parameters of $R$, i.e., $R/(x_1,\ldots,x_d)$ is of finite length as an $R$-module; see \cite[\S 14]{Mat89}.
It is elementary that the Zariski closed subset defined by the ideal $(x_1,\ldots,x_d)$  is just $\{\mm\}$, so that $\Gamma_{\mm}=\Gamma_{(x_1,\ldots,x_d)}$.
Thus we have $\RGamma_{\mm}X\cong \check{C}(\boldsymbol{x}) \otimes_R X$ in $\D(R)$ by \cref{Cech-comp}.
Hence, given an $R$-module $M$, we have $\H^i_\mm M=0$ for every $i>\dim R$ (\cite[Theorem 6.1.2]{BS13}).
\end{remark}

Now, let $R$ be a commutative noetherian ring, $\fa$ an ideal of $R$, $\boldsymbol{x}$ a system of generators of $\fa$, and $A$ a Noether $R$-algebra.

\begin{proposition}\label{Cech-comp-2}
Given  a complex $I$ of injective $A$-modules, we have a canonical chain map
$\Gamma_\fa I\to \check{C}(\boldsymbol{x}) \otimes_R I$ over $A$ which is a homotopy equivalence. 
\end{proposition}

\begin{proof}
By \cref{indec-gamma-0,indec-localize}, there is a degreewise split exact sequence
\[0\to \Gamma_{x_1} I\to I\to I_{x_1}\to 0\]
in the category of complexes of $A$-modules, where $\Gamma_{x_1}:=\Gamma_{(x_1)}$.
Then we obtain a canonical chain map $\Gamma_{x_1} I\to \check{C}(x_1)\otimes _RI$ over $A$ which is a homotopy equivalence.
An inductive argument yields a homotopy equivalence $\Gamma_{\fa}I\to \check{C}(\boldsymbol{x})\otimes_RI$ because $\Gamma_{\fa}=\Gamma_{x_n}\cdots\Gamma_{x_2}\Gamma_{x_1}$ as functors $\Mod A\to \Mod A$.
\end{proof}

Let $M$ be an $A$-module and $I$ an injective resolution of $M$ over $A$. Regarding $M$ and $I$ as objects in $\D(R)$, we have an isomorphism 
\begin{align}\label{A-resolution}
\RGamma_\fa M\cong \Gamma_\fa I
\end{align} in $\D(R)$ such that it induces an isomorphism $\H^i_\fa M\cong \H^i \Gamma_\fa I$ of $A$-modules for each $i\in \geq 0$. 
This fact, which is essentially proved in \cite[Corollary 2.7.2]{GN02}, follows from \cref{Cech-comp-2} along with \cref{Cech-comp}; cf. \cite[4.2.1]{BS13}.
The isomorphism \cref{A-resolution} will be important in the proof of \cref{idim}.

%%%%%%%%%%%%%%%%%%%%
%%%%%%%%%%%%%%%%%%%%
\subsection{Ideal-adic completion and projective modules over Noether algebras}\label{subsect-proj}

Let $R$ be a commutative noetherian ring and $\fa$ an ideal of $R$. 
Given a module $M$ over a Noether $R$-algebra $A$, we say that $M$ is $\fa$\emph{-complete} if the canonical map $M\to \Lambda^\fa M:=\varprojlim_{n\geq 1} M/\fa^n M$ is bijective.
When $R$ is local and $\mm$ is the maximal ideal, we say that $R$ is \emph{complete} if it is $\mm$-complete.

Let $(R,\mm)$ be a commutative noetherian local ring and $A$ a Noether $R$-algebra.
If $M\in \mod A$, there is a natural isomorphism 
$M\otimes_R \widehat{R}\cong \widehat{M}$ of right $A$-modules (see \cite[Theorem~8.7]{Mat89}), where $\Lambda^\mm:=\widehat{(-)}$.
Hence $\widehat{A}$ is finitely generated as an $\widehat{R}$-module. 
In addition, $\widehat{R}$ is a commutative noetherian (local) ring (see \cite[p.~63,~(4)]{Mat89} or \cite[Proposition~10.16]{AM69}).
Thus we can regard $\widehat{A}$ as a Noether $\widehat{R}$-algebra with structure map $\Lambda^\mm \varphi:\widehat{R}\to \widehat{A}$. 
We also have the canonical ring homomorphism $A\to \widehat{A}$. Note that $\widehat{A}$ is flat as a left and right $A$-module because $\widehat{R}$ is flat over $R$  (see \cite[Theorem~8.8]{Mat89}).

\begin{remark}\label{KS-PC}
Let $(R,\mm)$ be a commutative noetherian local ring and $A$ a Noether $R$-algebra.
Let $k$ be the residue field $R/\mm$.
By tensor-hom adjunction, $M^\vee:=\Hom_R(M,E_R(k))$ is a pure-injective left (resp. right) $A$-module for any right (resp. left) $A$-module $M$.

Assume $R$ is complete. Notice that then every finitely generated $A$-module is $\mm$-complete. Hence, for every finitely generated right $A$-module $M$, the canonical $A$-homomorphism $M\to M^{\vee\vee}$ is an isomorphism by Matlis duality (\cite[Theorem 3.2.13]{BH98}). Therefore $M\cong \Hom_R(M^{\vee},E_R(R/\mm))$ is pure-injective as a right $A$-module. 
It follows that every indecomposable finitely generated $A$-module has local endomorphism ring by \cite[Theorem 4.3.43]{Pre09} (see also the proof of \cite[Theorem 1.8]{LW12}).
As a consequence, $\mod A$ is a Krull--Schmidt category (see \cite[\S 4]{Kra15}). 
\end{remark}

\begin{lemma}\label{Warfield}
Let $R$ be a commutative noetherian complete local ring and $A$ a Noether $R$-algebra. Then every projective $A$-module is a direct sum of indecomposable (finitely generated) projective right $A$-modules.
\end{lemma}
\begin{proof}
Let $F$ be an arbitrary projective $A$-module. Then $F$ is a direct summand of some free $A$-module $A^{(I)}$. 
By \cref{KS-PC}, $A$ decomposes into a finite direct sum $\bigoplus_{1\leq i\leq n} P_i$ of indecomposable projective $A$-modules $P_i$ with local endomorphism ring. 
Thus $A^{(I)}= (\bigoplus_{1\leq i\leq n} P_i)^{(I)}$. It follows from \cite[Theorem 1]{War69} (see also \cite[Theorem E.1.26]{Pre09}) that $F$ is isomorphic to a direct sum of copies of $P_i$ for various $i$ with $1\leq i\leq n$.
\end{proof}

In fact, if $A$ is as in \cref{Warfield}, there is an isomorphism 
$A\cong \bigoplus_{P\in \Max A} T_A(P)^{n_P}$ with $T_A(P):=I_{A^\op}(P)^\vee$, where $n_P$ is the number appearing in \cref{ss-dec}; see \cite[Proposition 5.2]{KN22}.

%%%%%%%%%%%%%%%%%%%%%
%%%%%%%%%%%%%%%%%%%%%
\subsection{Complexes of flat modules}\label{subsect-flat}
Let $A$ be a left coherent ring for a while.
Then $\K(\Proj A)$ is compactly generated (\cref{cp}). 
Since the inclusion functor $\iota: \K(\Proj A)\to \K(\Flat A)$ commutes with small direct sums, it has a right adjoint 
\[\gamma: \K(\Flat A)\to \K(\Proj A)\]
 by the Brown representation theorem; see, e.g., \cite[Theorem 5.1.1]{Kra10}.
It follows from \cite[Facts 2.14]{Nee08} and a standard argument of Bousfield localization (see, e.g., \cite[Proposition 4.9.1]{Kra10}) that the counit morphism $\iota \gamma X \to X$ for each $X\in \K(\Flat A)$ induces a triangle
\begin{equation}\label{localization}
\iota \gamma X \to X\to Y\to \Sigma \iota \gamma X
\end{equation}
in $\K(\Flat A)$
such that $Y$ is a pure acyclic complex of flat right $A$-modules. Recall that a complex $Z$ of right $A$-modules is said to be \emph{pure acyclic} if $Z\otimes_AL$ is acyclic for any left $A$-module $L$.
If $X$ is a complex $X$ of flat right $A$-modules, then it is pure acyclic if and only if all the cycle modules $\Z^iX=\Ker(X^i \to X^{i+1})$ are flat.  Moreover, $\gamma X=0$ (i.e., $\iota \gamma X=0$) if and only if $X$ is pure acyclic by \cref{localization}; see \cite[Propositions 4.9.1, 4.10.1, and 4.12.1]{Kra10}.

Now let $R$ be a commutative noetherian ring and $A$ Noether $R$-algebra with structure map $\varphi: R\to A$. 
Let  $\pp\in \Spec R$. 
We write $\widehat{R_\pp}:=\Lambda^\pp (R_\pp)$ and  $\widehat{A_\pp}:=\Lambda^\pp (A_\pp)$.
Recall that the canonical ring homomorphism $R\to R_\pp \to \widehat{R_\pp}$ is flat, and there is a natural isomorphism $\widehat{R_\pp}\otimes_{R_\pp} A_\pp \cong \widehat{A_\pp}$.
Moreover, $\widehat{A_\pp}$ is a Noether $\widehat{R_\pp}$-algebra with structure map $\Lambda^\pp (\varphi \otimes_R R_\pp):\widehat{R_\pp}\to \widehat{A_\pp}$.

Let $\psi: A\to \widehat{A_\pp}$ be the canonical ring homomorphism.
We remark that $\widehat{A_\pp}$ is flat as a left and right $A$-module.
Consider the triangulated functor 
\[-\otimes_A \widehat{A_\pp}:\K_{\ac}(\Proj A)\to \K_{\ac}(\Proj \widehat{A_\pp}).\]
This commutes with small direct sums, and $\K_{\ac}(\Proj \widehat{A_\pp})$ is compactly generated (\cref{cp}). Hence the functor has a right adjoint 
\[\rho:\K_{\ac}(\Proj \widehat{A_\pp})\to \K_{\ac}(\Proj A)\]
by the Brown representation theorem. 

\begin{lemma}\label{pure-acyclic}
Let $X$ and $Y$ be acyclic complexes of finitely generated projective $ \widehat{A_\pp}$-modules. 
Then the canonical homomorphism
\[\Hom_{\K(\Proj \widehat{A_\pp})}(X,Y)\to \Hom_{\K(\Proj A)}(\rho X,\rho Y)\]
is bijective.
\end{lemma}

To prove this lemma, we make a remark.

\begin{remark}\label{pure-remark}
Since $\widehat{A_\pp}$ is flat as a left and right $A$-module, we have the following triangulated functors:
\begin{align*}
\psi^*=-\otimes_A \widehat{A_\pp}&:\K_\ac(\Flat A)\to \K_\ac(\Flat \widehat{A_\pp}),\\
\psi_*=\Hom_{\widehat{A_\pp}}(\widehat{A_\pp},-)&: \K_\ac(\Flat \widehat{A_\pp})\to \K_\ac(\Flat A).
\end{align*}
Notice that the scalar restriction functor $\psi_*$ is a right adjoint to the scalar extension functor $\psi^*$.
Let $Y\in \K_\ac(\Proj A)\subseteq \K_\ac(\Flat A)$ and $Z\in \K_\ac(\Proj \widehat{A_\pp})\subseteq \K_\ac(\Flat \widehat{A_\pp})$.
Then we have natural isomorphisms
\begin{align*}
\Hom_{\K(\Proj \widehat{A_\pp})}(Y\otimes_A \widehat{A_\pp}, Z)&=\Hom_{\K(\Flat \widehat{A_\pp})}(\psi^*Y, Z)\\
&\cong \Hom_{\K(\Flat A)}(Y,\psi_*Z)\\
&= \Hom_{\K(\Flat A)}(\iota Y,\psi_*Z)\\
&\cong\Hom_{\K(\Proj A)}(Y,\gamma\psi_*Z).
\end{align*}
where $\gamma\psi_*Z\in \K_{\ac}(\Proj A)$ by \cref{localization}.
By these isomorphisms, we see that
the composition
\[\K_{\ac}(\Proj \widehat{A_\pp})\xrightarrow{\subseteq}\K_{\ac}(\Flat \widehat{A_\pp})\xrightarrow{\psi_*}\K_{\ac}(\Flat A) \xrightarrow{\gamma}\K_{\ac}(\Proj A)\]
is a right adjoint to the functor
$-\otimes_A \widehat{A_\pp}:\K_{\ac}(\Proj A)\to \K_{\ac}(\Proj \widehat{A_\pp})$.
Hence we may interpret the above composition as $\rho$.

Let us further observe that 
the canonical homomorphism
\[\Hom_{\K(\Proj \widehat{A_\pp})}(X,Y)=\Hom_{\K(\Flat \widehat{A_\pp})}(X,Y)\to \Hom_{\K(\Flat A)}(\psi_* X,\psi_* Y)\]
 is bijective for $X, Y\in \K(\proj \widehat{A_\pp})$. It suffices to check that $\Hom_{\widehat{A_\pp}}(M,N)= \Hom_{A}(M,N)$ for $M, N\in \proj \widehat{A_\pp}$.
Since $M$ and $N$ are $\pp$-local $A$-modules, it easily follows that 
 $\Hom_{A_\pp}(M,N)=\Hom_A(M,N)$ (see, e.g., \cite[Remark 2.13]{KN22}).
In addition, since $M$ and $N$ are finitely generated as $\widehat{R_\pp}$-modules, these are $\pp$-complete. Then we have $\Hom_{\widehat{A_\pp}}(M,N)= \Hom_{A_\pp}(M,N)$; see \cite[Proposition A.14]{KN22} and its proof.
Therefore the desired equality $\Hom_{\widehat{A_\pp}}(M,N)= \Hom_{A}(M,N)$ follows.
\end{remark}

\begin{proof}[Proof of \cref{pure-acyclic}]
By \cref{pure-remark}, it suffices to show that 
the canonical homomorphism
\[\Hom_{\K(\Proj \widehat{A_\pp})}(X,Y)\to \Hom_{\K(\Proj A)}(\gamma \psi_* X,\gamma \psi_* Y)\]
is bijective. Since $X$ $Y$ are complexes of finitely generated projective $ \widehat{A_\pp}$-modules, 
the canonical homomorphism
$\Hom_{\K(\Proj \widehat{A_\pp})}(X,Y)\to \Hom_{\K(\Proj A)}(\psi_* X,\psi_* Y)$
is bijective by the remark again.
Hence it remains to show that 
the canonical homomorphism
\begin{align}\label{can-hom}
\Hom_{\K(\Flat A)}(\psi_* X,\psi_* Y)\to \Hom_{\K(\Proj A)}(\gamma\psi_* X,\gamma\psi_* Y)
\end{align}
is bijective.
By the adjoint pair $(\iota, \gamma)$,
there is a natural isomorphism \[\Hom_{\K(\Proj A)}(\gamma\psi_* X,\gamma\psi_* Y)\cong \Hom_{\K(\Flat A)}(\iota\gamma\psi_* X,\psi_* Y).\]
The composition of this isomorphism with \cref{can-hom} is the homomorphism
\begin{align}\label{can-hom-2}
\Hom_{\K(\Proj A)}(\psi_* X,\psi_* Y)\to \Hom_{\K(\Flat A)}(\iota\gamma\psi_* X,\psi_* Y)
\end{align}
induced by the counit morphism $\iota\gamma\psi_* X\to \psi_*  X$ (see, e.g., \cite[(1.5.6)]{KS06}).
Thus it suffice to show that \cref{can-hom-2} is bijective.
Since the mapping cone of the counit morphism $\iota\gamma\psi_* X\to \psi_*  X$ is pure acyclic by \cref{localization}, the proof will be completed once we show
$\Hom_{\K(\Flat A)}(Z,\psi_* Y)=0$
for every pure acyclic complex $Z$ of flat $A$-modules.
The complex $Y$ consists of finitely generated projective $\widehat{A_\pp}$-modules, so $\psi_*Y$ is a complex of flat cotorsion $A$-modules; see \cite[Proposition 5.2 and Theorem 5.6]{KN22}.
Then we have $\Hom_{\K(\Flat A)}(Z,\psi_* Y)=0$ by \cite[Theorem 1.3]{BCIE20}.
\end{proof}

The last equality in the proof can be also deduced by a similar argument to \cite[Remark 3.2]{NT20} (see also \cite[Proposition A.9]{KN22}). \cref{pure-acyclic} will be used in the proof of \cref{non-isolate} below.

For the reader's sake, let us mentioin that $\K(\Proj A)$ is well generated for any ring $A$ by Neeman  \cite[Theorem 1.1]{Nee08}, so the functor $\gamma: \K(\Flat A)\to \K(\Proj A)$ and the triangle \cref{localization} for each $X\in \K(\Flat A)$ exist.
Moreover, by \cite[Theorem 1.3]{BCIE20} and \cref{localization}, the canonical homomorphism
$\Hom_{\K(\Flat A)}(X,Y)\to \Hom_{\K(\Proj A)}(\iota \gamma X,Y)\cong \Hom_{\K(\Proj A)}(\gamma X,\gamma Y)$ is bijective whenever $X$ and $Y$ are complexes of flat cotorsion modules. In fact, $\gamma$ induces a triangulated equivalence from the homotopy category of complexes of flat cotorsion modules to the homotopy category of complexes of projective modules; see \cite[Theorem 1.2]{Nee08} and \cite[Remark A.9]{NT20}.

%%%%%%%%%%%%%%%%%%%%
%%%%%%%%%%%%%%%%%%%%
\subsection{Gorenstein orders and maximal Cohen--Macaulay modules}
\label{subset-Gorder}
Let $(R,\mm,k)$ be a commutative noetherian local ring. 
The \emph{depth} of an $R$-module $M$ is defined as
\[\depth_RM:=\inf\{i \mid \Ext^i_R(k,M)\neq 0\}.\]
By definition, $\{i \mid \Ext^i_R(k,M)\neq 0\}=\emptyset$ if and only if $\depth_RM=\infty$. Moreover, we have 
$\depth_RM=\inf\{i \mid \H^i_\mm M\neq 0\}$ (see \cite[Theorems 9.1]{ILL+07}). Since $\H^i_\mm M=0$ for $i>\dim R$ (\cref{Groth-vanish}), it holds that $\depth_{R}M\leq \dim R$ whenever $\depth_RM<\infty$.
In particular, the inequality $\depth_RM\leq \dim R$ holds for any non-zero finitely generated $R$-module $M$ (see \cite[Definition 1.46 and Theorem 9.3]{ILL+07}).

As in the introduction, we abbreviate Cohen--Macaulay to CM.
A \emph{maximal CM $R$-module} is a finitely generated $R$-module $M$ such that 
\[\depth_{R}M\geq \dim R.\]
Traditionally, maximal CM modules are assumed to be non-zero, but our definition allows the zero module to be maximal CM.
This convention is often used in the literature.
If a maximal Cohen--Macaulay $R$-module is non-zero, we call it \emph{small CM}.
The terminology comes from Hochster's ``small CM modules conjecture'' \cite[Conjecture (6) in p. 10]{Hoc75}.
Our usage of maximal CM modules and small CM modules is the same as Holm \cite{Hol17}.

Let $(R,\mm)$ be a CM local ring with a canonical module, that is, $R$ is a commutative noetherian local ring being small CM as an $R$-module and there exists a small CM $R$-module $\omega_R$ such that $\omega_R$ is a dualizing complex for $R$ (cf. \cite[Definition 3.3.1]{BH98} and \cite[Chapter V, Proposition 3.4]{Har66}).
The injective dimension of $\omega_R$ is $d:=\dim R$, and if $I$ is a minimal injective resolution of $\omega_R$, then $I^d\cong E_R(k)$; see, e.g., \cite[Theorem 3.1.17]{BH98} and \cite[Chapter V, Proposition 6.1]{Har66}. A canonical module is unique up to isomorphism (\cite[
Theorem 3.3.4]{BH98}.)

An \emph{$R$-order} is a Noether $R$-algebra $A$ such that $A$ is small CM as an $R$-module.
Given an $R$-order $A$,  we call a finitely generated $A$-module \emph{maximal CM} (resp.~\emph{small CM}) if it is maximal CM (resp.~small CM) as an $R$-module. We denote by $\CM A$ the category of maximal CM $A$-modules. 
There is an exact duality 
\begin{align}\label{duality}
\Hom_R(-,\omega_R):(\CM A)^\op\isoto \CM A^\op.
\end{align}
See \cite[Corollary 1.13]{Yos90} and \cite[p. 539]{IW14}.

Let $A$ be an $R$-order, and set $\omega_{A}:=\Hom_R(A,\omega_R)$, which is an $(A,A)$-bimodule  whose injective dimension is $d$ as a left and right $A$-module. 
We have
\begin{align*}
\CM A=&\{M\in \mod A \mid \Ext^i_R(M,\omega_R)=0~\forall i>0\}\\
=&\{M\in \mod A \mid \Ext^i_A(M,\omega_A)=0~\forall i>0\},
\end{align*}
where the first equality follows from local duality \cite[Theorem 3.5.8]{BH98} (see also \cite[Theorem 3.3.7]{BH98}).
Given an $A$-module $M$, we denote by $\add_A (M)$ the full subcategory of $\Mod A$ formed by  summands of finite direct sums of copies of $M$.
An $R$-order $A$ is called \emph{Gorenstein} if it satisfies one of the following equivalent conditions (\cite[Lemma 2.15]{IW14}):
\begin{enumerate}[label=(\roman*), font=\normalfont]
\item\label{right} $\omega_A$ is projective as a right $A$-module.
\item \label{left}  $\omega_A$ is projective as a left $A$-module.
\item \label{add} $\add_A (\omega_A)=\add_A (A)$.
\item \label{add-op} $\add_{A^\op} (\omega_{A})=\add_{A^\op} (A^\op)$.
\end{enumerate}
By definition, the injective dimension of a Gorenstein $R$-order $A$ is $d$ as a left and right $A$-module. In particular, $A$ is Iwanaga--Gorenstein (\cite[Definition~9.1.1]{EJ11}).
Moreover, $A_\pp$ is a Gorenstein $R_\pp$-order for each $\pp\in \Spec R$, and $\widehat{A_\pp}$ is a Gorenstein $\widehat{R_\pp}$-order; these facts follow from \cite[Theorems~7.11~and~17.3]{Mat89} and \cite[Corollary~2.1.8~and~Theorem~3.3.5]{BH98}.
Note that $R$ is Gorenstein as an $R$-order if and only if $R$ is a Gorenstein local ring (i.e., a commutative noetherian local ring of finite injective dimension), or equivalently, $R$ is isomorphic to $\omega_R$ (\cite[Theorem 3.3.7]{BH98}).
In general, an $R$-order $A$ being Iwanaga--Gorenstein may not imply that $A$ is Gorenstein as an $R$-order. 
Indeed, when $R$ is artinian, an $R$-order is Gorenstein if and only if $A$ is self-injective, that is, $A$ is a quasi-Frobenius $R$-algebra.

If $A$ is a Gorenstein $R$-order, then
\begin{align}\label{CM-Gproj}
\CM A=\{M\in \mod A \mid \Ext^i_A(M,A)=0~\forall i>0\}=\Gproj A
\end{align}
by \cref{add} and \cite[Corollary 11.5.3]{EJ11}.

\begin{remark}\label{proj-inj}
Let $A$ be an $R$-order, and interpret $\CM A$ as an exact category in a canonical way. Then $A$ and $\omega_A$ are a projective object and an injective object in $\CM A$, respectively. 
If $A$ is Gorenstein, then we have $\CM A=\Gproj A$ as observed above, so the projective objects and the injective objects in $\CM A$ coincide since this holds for the Frobenius category $\Gproj A$.
\end{remark}

Let $A$ an $R$-order.
Given $M,N\in \CM A$, write $\PH(M,N)$ for the set of morphisms factoring though finitely generated projective $A$-modules.
We denote by $\sCM A$ the \emph{stable category} of $\CM A$, that is,  the objects $\sCM A$ are those of $\CM A$, and the hom-set $\Hom_{\sCM A}(M,N)$ for $M,N\in \sCM A$ is defined as $\Hom_A(M,N)/\PH(M,N)$.
If $A$ is a Gorenstein $R$-order, then $\CM A=\Gproj A$, so 
$\sCM A=\sGproj A$ by definition.

Let $A$ be an $R$-order and $M\in \mod A$. Consider a projective resolution $\cdots  \to P^{-1}\to P^{0}\to M\to 0$ by finitely generated projective $A$-modules $P^i$. 
Set $\Omega^0 M:=M$, $\Omega^1 M:=\Ker(P^{0}\to M)$, and $\Omega^i M:=\Ker(P^{-i}\to P^{-i+1})$ for $i>0$.
Then $\Omega^i M$ is called an $i$th syzygy of $M$, which depends on only $M$, up to projective summands. 
Since $A$ is small CM as an $R$-module, we have $\depth_R P\geq d$ for 
every finitely generated projective $A$-module $P$. Thus a standard argument yields $\depth_R \Omega^i M\geq d$ (i.e., $\Omega^i M\in \CM A$) for every $i\geq d$; cf. \cite[Proposition 1.16]{Yos90}.

\begin{remark}\label{stable-hom}
Assume that $A$ is a Gorenstein $R$-order and let $M, N\in \CM A$.
Take an epimorphism $f:P\to N$ in $\mod A$ with $P$ projective. Consider the exact sequence 
$0\to \Omega^1N\to P\xrightarrow{f} N\to 0$.
This induces an exact sequence
\[0\to \Hom_A(M,\Omega^1N)\to \Hom_A(M,P) \to \Hom_A(M,N)\to \Ext^1_A(M,\Omega^1N)\to 0\]
by \cref{CM-Gproj}.
Since any morphism $F\to N$ from a finitely generated projective $A$-module $F$ factors through $f$, the image of the morphism $\Hom_A(M,P) \to \Hom_A(M,N)$ coincides with $\PH(M,N)$. 
Hence there is a canonical isomorphism
\begin{equation*}
\Hom_{\sCM A}(M,N)\cong \Ext^1_{A}(M,\Omega^1N)
\end{equation*}
of $R$-modules. 
This isomorphism will be use in the proof of \cref{Generator}.
\end{remark}

Let $A$ be an $R$-order $A$. We have $ \id_A A\geq \dim R$ by \cref{A-resolution} and \cite[Theorem 3.5.7(b)]{BH98}. Hence it also follows that $\gld A\geq \dim R$, where $\gld A$ denotes the global dimension of $A$.
Recall that $A$ is called \emph{non-singular} if $\gld A=\dim R$, or equivalently, if $\proj A=\CM A$ (\cite[Definition 1.6 and Proposition 2.17]{IW14}). In particular, every non-singular $R$-order is Gorenstein. Moreover, an $R$-order $A$ is non-singular if and only if $\widehat{A}$ is non-singular as an $\widehat{R}$-order; see, e.g., \cite[Proposition 2.26]{IW14}.
For convenience, we say that $A$ is \emph{singular} if $\gld A>\dim R$, or equivalently, if $\proj A\subsetneq \CM A$ .

We remark that an $R$-order $A$ is non-singular if and only if so is $A^\op$. Indeed, given a non-singular $R$-order $A$, we see from \cref{duality} that every $M\in \CM A^\op$ is an injective object in $\CM A^\op$. Since $A^\op$ is Gorenstein by definition, we have $\CM A^\op=\Gproj A^\op$, so the projective objects and the injective objects coincide (\cref{proj-inj}).Then it follows that $\proj A^\op=\CM A^\op$.

%%%%%%%%%%%%%%%%%%%%
%%%%%%%%%%%%%%%%%%%%
\subsection{Auslander--Reiten sequences}
\label{subsec-AR-tri}
Let $\cC$ be an additive category. A morphism $f: L\to M$ in $\cC$ is called \emph{left almost split} if $f$ is not a split monomorphism and given any morphism $a:L \to X$ that is not a split monomorphism, there is a morphism $b:M\to X$ such that $a=bf$.
Dually, a morphism $g: M\to N$ in $\cC$ is called \emph{right almost split} if $g$ is not a split epimorphism and given any morphism $a:X \to N$ that is not a split epimorphism, there is a morphism $b:X\to M$ such that $a=gb$.

Let $(R, \mm)$ be a complete CM local ring. Then $R$ admits a canonical module  (\cite[Corollary 3.3.8]{BH98}).
Let $A$ be an $R$-order. Recall that $\mod A$ is a Krull--Schmidt category (\cref{KS-PC}).
A non-split exact sequence $0\to L\xrightarrow{f} M\xrightarrow{g} N \to 0$ in $\CM A$ is called an \emph{Auslander--Reiten sequence}, or an \emph{AR-sequence}, if $f$ is left almost split and $N$ is indecomposable, or equivalently, if $g$ is right almost split and $L$ is indecomposable (cf. \cite[Chapter V, Proposition 1.14]{ARS97} and \cite[Exercise 13.32]{LW12}).
If $0\to L\xrightarrow{f} M\xrightarrow{g} N \to 0$ is an AR-sequence in $\CM A$, then we refer to it as an \emph{AR-sequence starting from $L$} or an \emph{AR-sequence ending in $N$}.

\begin{lemma}\label{Auslander-lemma}
Let $R$ be a complete CM local ring and $A$ an $R$-order. The following conditions are equivalent.
\begin{enumerate}[label=(\arabic*), font=\normalfont]
\item \label{locally-proj} $\proj A_\pp=\CM A_\pp$ for every $\pp\in \Spec R$ with $\pp\neq \mm$.
\item \label{AR-seq-1} Every indecomposable small CM $A$-module $L$ which is not an injective object in $\CM A$ has an AR-sequence starting from $L$.
\item \label{AR-seq} Every indecomposable small CM $A$-module $N$ which is not a projective object in $\CM A$ has an AR-sequence ending in $N$.
\end{enumerate}
\end{lemma}
\begin{proof}
The equivalence \cref{locally-proj}$\Leftrightarrow$\cref{AR-seq} is essentially shown in \cite[Theorem 2.1(a)]{AR87}. See \cref{TE-1}\cref{trivial-ext} below.

Suppose \cref{AR-seq} holds. By the equivalence \cref{locally-proj}$\Leftrightarrow$\cref{AR-seq},
$A_\pp$ is non-singular for every $\pp\in \Spec R$ with $\pp\neq \mm$, and hence $A_\pp^\op$ is non-singular for every $\pp\in \Spec R$ with $\pp\neq \mm$. Thus $\proj A_\pp^\op=\CM A_\pp^\op$ for every $\pp\in \Spec R$ with $\pp\neq \mm$, that is, \cref{locally-proj} holds for the $R$-order $A^\op$. 
Hence \cref{AR-seq} holds for $A^\op$.
Then \cref{AR-seq-1} holds for $A$ by the duality \cref{duality}. We have shown \cref{AR-seq}$\Rightarrow$\cref{AR-seq-1}.
Conversely, suppose \cref{AR-seq-1} holds.
Then \cref{AR-seq} holds for $A^\op$ by the duality \cref{duality}. 
Hence $A^\op_\pp$ is non-singular for $\pp\in \Spec R$ with $\pp\neq \mm$, or equivalently, 
 $A_\pp$ is non-singular for $\pp\in \Spec R$ with $\pp\neq \mm$. 
Thus \cref{locally-proj} holds for $A$. Then \cref{AR-seq} holds for $A$.
We have shown \cref{AR-seq-1}$\Rightarrow$\cref{AR-seq}.
\end{proof}

\begin{remark}\label{TE-1}
\begin{enumerate}[label=(\arabic*), font=\normalfont, wide, labelwidth=!, labelindent=0pt]
\item Let $(R,\mm)$ and $(S,\nn)$ be commutative noetherian local rings such that $\dim R=\dim S$. Let $f:S\to R$ be a ring homomorphism by which $R$ is a Noether $S$-algebra. 
Then $R/\nn R$ is a artinian local ring, so the Zariski closed subset of $\Spec R$ defined by $\nn R$ is just $\{\mm\}$. Hence $\Gamma_{\nn R} = \Gamma_{\mm}$ as functors $\Mod R\to \Mod R$.
Let $M$ be an $R$-module $M$, which we can also regard as an $S$-module.
Then we have a natural isomorphism $\H^i_\nn M \cong  \H^i_{\nn R} M=\H^i_\mm M$ of $S$-modules for every $i\geq 0$ by the independence theorem of local cohomology; see \cite[4.2.1]{BS13}.
It follows from this fact that a finitely generated $R$-module is maximal CM if and only if it is maximal CM over $S$. In fact, the isomorphism $\H^i_\nn M\cong \H^i_\mm M$ can be recovered by \cref{A-resolution}\footnote{This does not mean that the full generality of the independence theorem is covered by \cref{A-resolution}.}.
\label{local-hom}

\setlength{\itemindent}{12pt}
\item \cite[Theorem 2.1(a)]{AR87}\footnote{``Equidimensional'' in the sense of \cite{AR87} (see \cite[p. 80]{Aus78})  is different from that of \cite{Mat89} and the same meaning as ``coequidimensional'' in the sense of \cite{IW14}.} treats an order over a complete Gorenstein local ring, but  the complete Gorenstein local ring can be replaced by a complete CM local ring. Instead of modifying the proof of \cite[Theorem 2.1(a)]{AR87}, we apply a standard technique using trivial extensions.  For details about trivial extensions, see, e.g., \cite[Remark 11.40 and Theorem 11.42]{ILL+07}.
Recall also that every Gorenstein local ring is CM (\cite[Proposition 3.1.20]{BH98}).

First, let $R$ be a CM local ring with a canonical module $\omega_R$, $S$ the trivial extension of $R$ by $\omega_R$, and $f:S\to R$ the canonical ring homomorphism which is surjective.
It is well known that $S$ is a Gorenstein local ring with $\dim R= \dim S$.  Hence, by \cref{local-hom}, a finitely generated $S$-module is maximal CM if and only if it is maximal CM over $R$.
Next, let $A$ be an $R$-order. Since $f$ is surjective, we have $A_{\pp}=A_{f^{-1}(\pp)}$ for every $\pp\in \Spec R$, and $A_{\qq}=0$ if $\qq\in \Spec S$ and $\{\pp\in \Spec R \mid f^{-1}(\pp)=\qq\}=\emptyset$.
Finally, assume $R$ is complete. Then $S$ is complete because there is also a canonical ring homomorphism $R\to S$ by which we can regard $S$ as a finitely generated $R$-module. 
Then $S$ is a complete Gorenstein local ring and $A$ is an $S$-order. It is easily seen from the above observation that each condition of \cref{Auslander-lemma} for the $R$-order $A$ is equivalent to that of this lemma for the $S$-order $A$.
\label{trivial-ext}

\end{enumerate}
\end{remark}

Let $R$ be a CM local ring with a canonical module. An $R$-order $A$ is said to have \emph{at most an isolated singularity} if, for any $\pp\in \Spec R$ with $\pp\neq \mm$, $A_\pp$ is non-singular.
If, conversely, there exists $\pp\in \Spec R$ with $\pp\neq \mm$ such that $A_\pp$ is singular, then $A$ is said to have a \emph{non-isolated singularity}.
We say that $\CM A$ \emph{has AR-sequences} if \cref{Auslander-lemma}\cref{AR-seq-1} holds, or equivalently, if \cref{Auslander-lemma}\cref{AR-seq} holds. 
The next theorem is well known to experts.

\begin{theorem}\label{Auslander}
Let $R$ be a complete CM local ring and $A$ an $R$-order.
The following conditions are equivalent:
\begin{enumerate}[label=(\arabic*), font=\normalfont]
\item \label{isolate} A has at most an isolated singularity.
\item \label{fl} For all $M,N\in \CM A$, $\Hom_{\sCM A}(M,N)$ is of finite length as an $R$-module.
\item \label{has-AS-sequence}  $\CM A$ has AR-sequences.
\end{enumerate}
\end{theorem}

\begin{proof}
The equivalence \cref{isolate}$\Leftrightarrow$\cref{has-AS-sequence} follows from \cref{Auslander-lemma}. 
The equivalence \cref{locally-proj}$\Leftrightarrow$\cref{fl} follows from \cite[Chapter I, Lemma 7.6(d)]{Aus78}.
\end{proof}

Let $R$ be a complete CM local ring and $A$ an $R$-order. 
We say that $A$ is of \emph{finite CM representation type} if there are, up to isomorphism, only finitely many indecomposable small CM $A$-modules. 
The next theorem is also well known to experts (cf. \cite[Theorem 4.22]{Yos90}).

\begin{theorem}\label{isolate-sing}
Let $R$ be a complete CM local ring and $A$ an $R$-order.
If $A$ is of finite CM representation type, then $A$ has at most an isolated singularity.
\end{theorem}
\begin{proof}
This is  proved in \cite[\S 10, Theorem]{Aus86} assuming $R$ is regular, but the same proof works in our setup.
\end{proof}

\cref{Auslander-lemma,Auslander} are also proved in \cite{Aus86} assuming $R$ is regular. We remark that  if $R$ is a regular local ring and $A$ is an $R$-order, then a finitely generated $A$-module is maximal CM if and only if it is free over $R$ (cf. \cite[Proposition 1.9]{Yos90}). In particular, $A$ is free over $R$ in this case.
It is often but not always possible to replace a given $R$-order $A$ over a complete CM local ring $R$ by an order over a complete regular local ring;  see the remark below.

\begin{remark}\label{regular-loc}
Let $R$ be a commutative noetherian complete local ring, and assume that $R$ contains a filed or $R$ is an integral domain.
Then there exists a subring $S$ of $R$ such that $S$ is a complete regular local ring and the injection $f: S\hr R$ makes $R$ a Noether $S$-algebra, where $\dim R=\dim S$; see \cite[Remark in p. 215 and Theorem 29.4(iii)]{Mat89} and \cite[Corollary A.8]{BH98}. 

Let $f$ be as above. Assume that $R$ is CM and let $A$ be an $R$-order. Using $f$, we naturally regard $A$ as a Noether $S$-algebra, and then $A$ is small CM as an $S$-module by \cref{TE-1}\cref{local-hom}. Hence $A$ is an $S$-order. Notice from \cref{TE-1}\cref{local-hom} that the maximal CM modules over the $R$-order $A$ coincide with the maximal CM modules over the $S$-order $A$.

In general, there exists a complete Gorenstein local ring such that it cannot be an order over any complete regular local ring. Such an example can be found in \cite[Remark in p. 226]{Mat89}.
\end{remark}

\begin{remark}\label{seq-tri}
Let $R$ be a CM local ring with a canonical module and $A$ a Gorenstein $R$-order.
Then $\sCM A=\sGproj A$ (\cref{proj-inj}), so the stable category $\sCM A$ is triangulated and we have the canonical functor $\sCM A=\sGproj A\hr \sGProj A$.

If $0\to L\to M\to N\to 0$ is an exact sequence in $\CM R$, we have a triangle in $\sCM A$ of the form $L\to M\to N\to \Sigma L$. 
Moreover, if the sequence $0\to L\to M\to N\to 0$ is an AR-sequence, the triangle $L\to M\to N\to \Sigma L$ is an AR-triangle. See Section~\ref{AR-triangles} for the definition of AR-triangles.
\end{remark}

For a compactly generated triangulated category $\cT$, we denote by $\cT^\cp$ the full subcategory of compact objects of $\cT$.

\begin{theorem}\label{comp-Gor-case}
Let $R$ be a complete CM local ring and $A$ a Gorenstein $R$-order. The following statements hold:
\begin{enumerate}[label=(\arabic*), font=\normalfont]
\item \label{equiv-comp} $\sGProj A$ is compactly generated and the canonical functor $\sCM A=\sGproj A\to \sGProj A$ induces a triangulated equivalence  $\sCM A\isoto (\sGProj A)^\cp$.
\item \label{AR-stab} Assume that $A$ has at most an isolated singularity. Then every indecomposable object $L\in (\sGProj A)^\cp$ has an AR-triangle starting from $L$.
\end{enumerate}
\end{theorem}

\begin{proof}
Since $A$ is Iwanaga--Gorenstein, $\sGProj A$ is compactly generated, and we have the fully faithful triangulated functor $\sCM A=\Gproj A\hr (\GProj A)^\cp \subseteq \GProj A$. Moreover, given $M\in (\GProj A)^\cp$, there exists $N\in \sCM A$ such that  $N\cong M\oplus L$ in $(\sGProj A)^\cp$ by some $L\in (\sGProj A)^\cp$. See \cref{compact}.
We want to show that $L$ belongs to $\sCM A$, up to isomorphism.
Note that $N$ decomposes into a finite direct sum $\bigoplus_{1\leq i\leq n}N_i$ of indecomposable objects $N_i\in \sCM A$ with local endomorphism ring; see \cref{KS-St} below. 
Thus $\bigoplus_{1\leq i\leq n}N_i\cong M\oplus L$ in $\sGProj A$.
Since $\sGProj A$ is a triangulated category with small direct sums, any idempotent morphism in $\sGProj A$ splits (see \cite[Proposition 1.6.8]{Nee01}).
Then we can easily deduce from \cite[Lemma 1.2]{LW12} that $M$ is in $(\sGProj A)^\cp$ isomorphic to $\bigoplus_{j\in J} N_j$ for some subset $J\subseteq \{1,\ldots,n\}$. 
We have shown \cref{equiv-comp}.

\cref{AR-stab} follows from \cref{equiv-comp,Auslander,seq-tri}.
\end{proof}

\begin{remark}\label{KS-St}
Let $A$ be as in \cref{comp-Gor-case}.
For an indecomposable module $M\in \CM A$, $\End_A(M)$ is a (possibly noncommutative) local ring by \cref{KS-PC}. 
Furthermore, the canonical ring homomorphism $\End_A(M)\to \End_{\sCM A}(M)$ is surjective, so $\End_{\sCM A}(M)$ is a local ring whenever $M$ is not a projective $A$-module.
Since the objects of $\sCM A$ are those of $\CM A$ and the canonical functor $\CM A\to \sCM A$ is additive, given a nonzero object $N\in \sCM A$, it can be in $\CM A$ decomposed as a finite direct sum $N=\oplus_{1\leq i\leq n}N_i$ of indecomposable modules $N_i$ in $\CM A$, and then we have $N= \oplus_{1\leq i\leq n}N_i$ in $\sCM A$, where each $N_i\in \sCM A$ has local endomorphism ring. In particular, $\sCM A$ is a Krull--Schmidt category.
\end{remark}

%%%%%%%%%%%%%%%%%%%%%%%%%%
%%%%%%%%%%%%%%%%%%%%%%%%%%

\section{Proof of the main theorem}
\label{Proof}

We start with the following lemma.

\begin{lemma}\label{idim}
Let $R$ be a commutative noetherian local ring and $A$ a Noether $R$-algebra. Let $M$ be a nonzero finitely generated $A$-module. Assume that $\Ext^n_A(A/\rad A,M)=0$ for some $n\geq\dim R$. Then $\id_A M\leq n-1$.
\end{lemma}

\begin{proof}
Let $E=(0\to E^0\to E^1\to \cdots)$ be a minimal injective resolution of $M$ over $A$.
We have a canonical isomorphism $\RGamma_\mm M\cong \Gamma_{\mm}E$ in $\D(R)$ by \cref{A-resolution}. 
Since $\H^i_\mm M=0$ for every $i>d:=\dim R$ (\cref{Groth-vanish}), the sequence
\[\Gamma_\mm E^d \to \Gamma_\mm E^{d+1}\to  \Gamma_\mm E^{d+2}\to \cdots\]
 is exact, where each $\Gamma_\mm E^i$ is injective by \cref{indec-gamma}. 
Furthermore, $\Gamma_\mm E^n=0$ as $\Hom_A(A/\rad A,E^n)=\Ext^n_A(A/\rad A,M)=0$ by assumption and \cref{injective}.
Since $n\geq d$, we obtain the exact sequence
\[0 \to \Gamma_\mm E^{n+1}\to  \Gamma_\mm E^{n+2}\to\cdots,\]
which splits in $\Mod A$.
Thus
the induced complex 
\[0 \to \Hom_A(A/\rad A, \Gamma_\mm E^{n+1})\to  \Hom_A(A/\rad A,\Gamma_\mm E^{n+2}) \to \cdots\]
splits as well.
However this complex has zero differential, because
 it is a truncated complex of $\Hom_A(A/\rad A, \Gamma_\mm E)\cong \Hom_A(A/\rad A, E)$; see \cref{injective}. 
It follows that $\Ext_A^i(A/\rad A, M)\cong \Hom_A(A/\rad A, E^{i})=0$ for every $i> n$.
Since $\Ext^n_A(A/\rad A,M)=0$ by assumption, we have
\[\sup\{i\mid \Ext^i_A(A/\rad A,M)\neq 0\}\leq n-1.\]
Then $\id_A M \leq n-1$ by \cref{id-test}.
\end{proof}

\begin{remark}\label{Bass-conjecture}
Let $R$, $A$, $M$, and $n$ be as in \cref{idim}. 
Then $n$ cannot be zero. Indeed, if $n$ is zero, then $\dim R$ must be zero, so $A$ is an Artin $R$-algebra; in this case the equality $\Ext^0_A(A/\rad A,M)=0$ implies $M=0$ (see \cref{injective}), but this contradicts the assumption that $M$ is nonzero.

On the other hand, if $A$ is projective as an $R$-module, then $R$ is CM and $n$ is greater than $\dim R$.
This fact follows from the validity of ``Bass conjecture'' (see \cite[Theorem 3.1.17, Corollary 9.6.2, and Remark 9.6.4(ii)]{BH98}) because every injective right $A$-module is an injective $R$-module by the standard isomorphism $\Hom_A(-\otimes_RA,I)\cong \Hom_R(-,\Hom_A(A,I))$.
When $A$ is not projective as an $R$-module, it is unclear if the same fact on $R$ and $n$ holds or not; cf. \cite[Question 3.11]{GN02}.
\end{remark}

The author obtained the following result thanks to a suggestion by Ryo Takahashi.

\begin{proposition}\label{Generator}
Let $R$ be a CM local ring with a canonical module and $A$ a Gorenstein $R$-order. 
Set $d:=\dim R$. Then $\Hom_{\sCM A}(\Omega^d(A/\rad A),N)\neq 0$ for every nonzero object $ N\in \sCM A$.
\end{proposition}

\begin{proof}
Let $0\neq N\in \sCM A$, and suppose that $\Hom_{\sCM A}(\Omega^d(A/\rad A),N)=0$.
By \cref{stable-hom}, we have $\Ext^1_{A}(\Omega^d(A/\rad A),\Omega^1N)=0$.
In addition, there are isomorphisms
\begin{equation*}
\Ext^1_{A}(\Omega^d(A/\rad A),\Omega^1N)\cong \Ext^2_{A}(\Omega^{d-1}(A/\rad A),\Omega^1N) \cong \cdots \cong
 \Ext^{d+1}(A/\rad A,\Omega^1N).
\end{equation*}
Thus $\Ext^{d+1}(A/\rad A,\Omega^1N)=0$. This implies that $\Omega^1N$ has finite injective dimension by \cref{idim}. Then $\Omega^1N$ has finite projective dimension since $A$ is Iwanaga--Gorenstein (see \cite[Theorem 9.1.10]{EJ11}).
Thus $N$ also has finite projective dimension. 
Since $N\in \CM A=\Gproj A$ by \cref{CM-Gproj}, it follows that $N$ is projective, i.e., $N=0$ in $\sCM A$. This is a contradiction. 
\end{proof}

\begin{lemma}\label{pure-inj}
Let $R$ be a CM local ring with a canonical module and $A$ a Gorenstein $R$-order. 
Assume that $M\in \GProj A$ is pure-injective in $\Mod A$. Then $M$ is pure-injective in $\sGProj A$

In particular, if $R$ is complete, then every object of $\sCM A$ is pure-injective in $\sGProj A$.
\end{lemma}

\begin{proof}
Since triangulated category $\sGProj A$ is compactly generated (\cref{compact}), it has small direct sums and small direct products (see \cite[Lemma 1.5]{Kra00}).
Let $I$ be a small set and let $\phi:M^{(I)}\to M^I$ be the canonical morphism in $\sGProj A$.
It directly follows from the definition of $\sGProj A$ that the canonical functor $G:\GProj A\to \sGProj A$ commutes with small direct sums; see also \cref{diagram}.
By the definition of $\sGProj A$ again, there exist $N\in \GProj A$ and a morphism $f: M^{(I)}=\bigoplus_{i\in I}M_i\to N$ in $\GProj A$ such that $G(f)=\phi$, where $M_i:=M$.
For each $j\in I$, we have the canonical injection $\varepsilon_j:M_j\to \bigoplus_{i\in I}M_i$ and the canonical projection $\pi_j:\prod_{i\in I}M_i\to M_j$ in $\sGProj A$. 
Clearly, the canonical injection $e_j:M_i\to \bigoplus_{i\in I}M_i$ in $\GProj A$ satisfies $G(e_j)=\varepsilon_j$. Moreover, there exists a morphism $p_j:N\to M_j$ in $\GProj A$ such that $G(p_j)=\pi_j$.
By definition, the composition 
\[M_j\xrightarrow{\varepsilon_j} \bigoplus_{i\in I}M_i\xrightarrow{\phi} \prod_{i\in I}M_i \xrightarrow{\pi_i}M_j\]
 is the identity $\underline{\id}_{M_j}: M_j\to M_j$ in $\sGProj A$, and so $G(p_jfe_j)=\pi_j \phi\varepsilon_j=\underline{\id}_{M_j}$.
 Then there exist morphisms $M_j\xrightarrow{g_j} Q_j\xrightarrow{h_j} M_j$ in $\GProj A$ such that $Q_j$ is a projective $A$-module and $p_j f e_j+h_j g_j$ equals the identity  $\id_{M_j}:M_j\to M_j$  in $\GProj A$.

Now, denote by $(g_i)_{i\in I}$ the morphism $\bigoplus_{i\in I}M_i\to \bigoplus_{i\in I}Q_i$ induced by $g_i: M_i\to Q_i$ for each $i\in I$.
The composition
\begin{equation}\label{pure-mono}
M_j\xrightarrow{\ \, e_j\ \,} \bigoplus_{i\in I}M_i\xrightarrow{\scriptsize\begin{pmatrix}f \\ (g_i)_{i\in I}\end{pmatrix}} N\oplus (\bigoplus_{i\in I}Q_i)\xrightarrow{\scriptsize\begin{pmatrix}p_j&h_j\end{pmatrix}} M_j.\end{equation}
in $\GProj A$ is $\id_{M_j}$ as $p_j f e_j+h_j g_j=\id_{M_j}$. Then it easily follows that the second morphism of \cref{pure-mono} is a pure monomorphism in $\Mod A$.
Since $M$ is pure-injective in $\Mod A$, $\Hom_A(-,M)$ sends the pure monomorphism to a surjection. Therefore the summation morphism $s: \bigoplus_{i\in I}M_i\to M$ in $\GProj A$ can be written as the composition of the second morphism of \cref{pure-mono} and some morphism $u:N\oplus (\bigoplus_{i\in I}Q_i)\to M$.
Send $s$, $u$, and \cref{pure-mono} to $\sGProj A$ by the canonical functor $G: \GProj A\to \sGProj A$.
Then we obtain the following commutative diagram:
\[
\begin{tikzcd}
M_j \ar[r,"\varepsilon_j"]&\displaystyle{\bigoplus_{i\in I}M_i}\ar[rr,"\phi"]\arrow[rd,"G(s)"'] &&\displaystyle{\prod_{i\in I}M_i}\ar[r,"\pi_j"]\ar[ld,"G(u)"]&M_j\\
&& M
\end{tikzcd}
\]
Since $G(s)$ is the summation morphism $\bigoplus_{i\in I}M_i\to M$ in $\sGProj A$, 
the above diagram along with \cref{injectives} shows that $M$ is pure-injective in $\sGProj A$.

The second claim of the lemma follows from the first claim and \cref{KS-PC}.
\end{proof}

\begin{proposition}\label{non-isolate}
Let $R$ be a complete CM local ring and $A$ a Gorenstein $R$-order.
Assume that $A$ has a non-isolated singularity. Then there exists a non-compact indecomposable pure-injective object in $\sGProj A$.
\end{proposition}

\begin{proof}
By assumption, there exists a non-maximal prime ideal $\pp\in \Spec R$ such that $A_\pp$ is singular as an $R_\pp$-order, or equivalently, $\widehat{A_\pp}$ is singular as an $\widehat{R_\pp}$-order.
Hence there exists an indecomposable small CM module $M$ over $\widehat{A_\pp}$ such that $M\notin \proj  \widehat{A_\pp}$. 
Since $\End_{\sCM \widehat{A_\pp}}(M)=\End_{\sGProj \widehat{A_\pp}}(M)$ is a local ring (\cref{KS-St}), $M$ is indecomposable in $\sGProj \widehat{A_\pp}$.
Furthermore, $M$ is pure-injective in $\sGProj \widehat{A_\pp}$ by \cref{KS-PC,pure-inj}.
Let $X\in \K_{\ac}(\proj \widehat{A_\pp})$ be a complete resolution of $M$.
By the triangulated equivalence $\sGProj \widehat{A_\pp}\cong \K_{\ac}(\Proj \widehat{A_\pp})$, $X$ is pure-injective in $\K_{\ac}(\Proj \widehat{A_\pp})$.
Then the functor $\rho: \K_{\ac}(\Proj \widehat{A_\pp})\to \K_{\ac}(\Proj A)$ defined in \cref{subsect-flat} sends $X$ to a pure-injective object $\rho X$ in $\K_{\ac}(\Proj A)$; see \cref{pure-right}.
Since we have the triangulated equivalence $\K_{\ac}(\Proj A)\cong \sGProj A$, it remains to show that $\rho X$ is non-compact and indecomposable.

By \cref{pure-acyclic}, there is a natural isomorphism
\[\End_{\K(\Proj \widehat{A_\pp})}(X)\isoto \End_{\K(\Proj A)}(\sigma X)\] of rings, and 
$\End_{\sGProj \widehat{A_\pp}}(M)\cong \End_{\K(\Proj \widehat{A_\pp})}(X)$ is a local ring.
Thus $\sigma X$ is indecomposable in $\K_\ac(\Proj A)$.
Moreover,  $\End_{\K(\Proj \widehat{A_\pp})}(X)$ is a nonzero $R_\pp$-module and the functor $\rho$ is $R$-linear. Therefore $\End_{\K(\Proj A)}(\sigma X)$ is a nonzero $R_\pp$-module.
Then $\sigma X$ cannot be compact in $\K_\ac(\Proj A)$ by Nakayama's lemma, because for every $Y\in \K_\ac(\Proj A)^\cp$, $\End_{\K(\Proj A)}(Y)$ is finitely generated as an $R$-module by the equivalence
$\K_\ac(\Proj A)^\cp\cong \sCM A$; see \cref{diagram} and 
\cref{comp-Gor-case}\cref{equiv-comp}.
\end{proof}

We are now ready to prove the main theorem.

\begin{proof}[Proof of \cref{main-theorem}.]
We first show the implication \cref{fCM}$\Rightarrow$\cref{GP}. Suppose \cref{fCM} holds. 
Thus there are only finitely many indecomposable small CM $A$-modules $M_1,\ldots,M_n$ up to isomorphism. 
Consider a sequence 
\begin{align}\label{sequence}
X_1\xrightarrow{f_1} X_2\xrightarrow{f_2} X_2 \xrightarrow{f_3} \cdots
\end{align}
of indecomposable compact objects in $\sGProj A$, and suppose that each $f_i$ is not an isomorphism.
Since $\sCM A\cong(\sGProj A)^\cp$ by \cref{comp-Gor-case}\cref{equiv-comp},
without loss of generality, we may assume that there exists an infinite set $S$ of positive integers $i$ such that $M_1= X_i$ in $\sGProj A$ for all $i\in S$; see also \cref{KS-St}. 
Since $E:=\End_{\sCM A}(M_1)$ is a local ring and $A$ has at most an isolated singularity by assumption,  $E$ is of finite length as an $R$-module; see \cref{Auslander}. 
Hence $E$ is an artinian local ring. 

Let $J$ be the maximal ideal of $E$. Then $J$ consists of the non-isomorphisms from $M_1$ to $M_1$. Moreover, there exists an integer $t\geq 1$ with $J^t=0$.
Since $S$ is an infinite set, we can take some integers $s_1, s_2,\ldots, s_t\in S$ with  $0\leq s_1<s_2<\cdots <s_t<s_{t+1}$.
For each $1\leq j\leq t$, let $g_j$ be the composition of all morphism between $X_{s_j}$ and $X_{s_{j+1}}$ in \cref{sequence}. 
Then we obtain the following sequence
\[X_{s_1}\xrightarrow{g_1}X_{s_2}\xrightarrow{g_2}X_{s_3}\xrightarrow{g_3}\cdots \xrightarrow{g_{t}}X_{s_{t+1}}.\]
By the definition of $S$, we have $X_{s_1}=M_1$ for each $1\leq j\leq t$.
Furethermore, every $g_j$ is not an isomorphism because for every $i\geq 0$, $f_i$ in \cref{sequence} is not an isomorphism and $X_i$ is indecomposable.
As a consequence, we have $g_j\in J$ for every $1\leq j\leq t$.
Thus $g_t \cdots g_1\in J^t=0$, and then \[g_t \cdots g_1=f_{{s_{t+1}}-1} \cdots f_{s_1}\] is the zero map. Therefore, by \cite[Theorem 2.10]{Kra00} or \cite[Theorem 9.3]{Bel00}, each object $Y\in \sGProj A$ has a decomposition $Y=\bigoplus_{\lambda\in \Lambda}Y_\lambda$ in $\sGProj A$ with $Y_\lambda \in \sCM A=\sGproj A\cong (\sGProj A)^\cp$ for each $\lambda\in \Lambda$.
Then there exist projective $A$-modules $P$ and $Q$ such that $Y\oplus P\cong (\bigoplus_{\lambda\in \Lambda}Y_\lambda) \oplus Q$ in $\GProj A$.
It follows from \cref{Warfield} and \cite[Theorem 1]{War69} that $Y$ is isomorphic to a direct sume of finitely generated Gorenstein-projective $A$-modules.
Thus \cref{GP} holds.

The implication \cref{GP}$\Rightarrow$\cref{cCM} follows from \cref{comp-Gor-case}\cref{equiv-comp}.

Finally, we prove the implication \cref{cCM}$\Rightarrow$\cref{fCM}. 
Suppose $\cref{cCM}$ holds. Then $A$ has at most an isolated singularity by \cref{non-isolate}.
Set $\cT:=\sGProj A$ and $d:=\dim R$.
By the triangulated equivalence $\sCM A\isoto \cT^\cp$, we may regard $F:=\Hom_{\sCM A}(\Omega_d(A/\rad A),-)$ as a coherent functor $\cT\to \Ab$; see \cref{TopZiegler}.
Now, suppose that $\CM A$ has infinitely many indecomposable objects up to isomorphism; then so does $\sCM A$ (see \cref{KS-St}).
Thus we see from \cref{KS-PC,pure-inj} that $\cT^\cp$ contains,  up to isomorphism, infinitely many indecomposable pure-injective objects in $\cT$.
By \cref{Generator}, the isomorphism classes of all such objects belong to the open set $(F)$ of the Ziegler spectrum $\Zg_\cT$ (see \cref{TopZiegler}).
Therefore,$(F)$ contains, up to isomorphism, infinitely many objects in $\cT^\cp$.
Then there exists a non-compact indecomposable pure-injective object in $\cT=\sGProj A$ by \cref{comp-Gor-case}\cref{AR-stab} and \cref{Existence}, but this is a contradiction. Hence $A$ is of finite CM representation type, that is, \cref{fCM} holds.
\end{proof}

%%%%%%%%%%%%%%%%%%%%
%%%%%%%%%%%%%%%%%%%%
\subsection*{Acknowledgments}
The author is grateful to Rosanna Laking for valuable discussions and her excellent appendix.
The author would also like to thank Osamu Iyama and Ryo Takahashi for helpful comments and suggestions.

This project was started when the author was a postdoc at the University of Verona. Part of the project was completed while he was a JSPS Research Fellow at the Nagoya University and later at the University of Tokyo.

%%%%%%%%%%%%%%%%%%%%%%
\appendix
%%%%%%%%%%%%%%%%%%%%%%%
\section{The Ziegler spectrum}

In this appendix we use the Ziegler spectrum of a compactly generated category to lay the groundwork for the proof of $(3)\Rightarrow(1)$ in \cref{main-theorem}.  Namely we provide a sufficient condition for the existence of a non-compact indecomposable pure-injective object in a compactly generated triangulated category.

Our approach is inspired by Prest's observation that \cref{Artin} can be proved using various topological properties of the Ziegler spectrum of an Artin algebra \cite[\S 5.3.4]{Pre09}.  For example, the proof of $(3)\Rightarrow(1)$ of \cref{Artin} can be viewed as a consequence of the fact that the Ziegler spectrum is quasi-compact and that every finitely generated module corresponds to an isolated point of the spectrum.  This point of view is also explained in \cite[\S 1.7]{Lak} in the case of a finite-dimensional algebra.  The crucial difference between the Ziegler spectrum of a module category and that of a compactly generated triangulated category is that the latter is not, in general, quasi-compact (see \cite[Theorem 17.3.22]{Pre09}). In this appendix, we will adapt Prest's argument to our setting by working within a quasi-compact open set.

\subsection{The points of the Ziegler spectrum}
\label{PointsZiegler}
Let $\cT$ be a compactly generated triangulated category with the full subcategory of compact objects of $\cT$ denoted by $\cT^\cp$ (see \cite[\S 5.3]{Kra10}).  The category $((\cT^\cp)^\op, \Ab)$ of additive contravariant functors from $\cT^\cp$ to the category $\Ab$ of abelian groups is a locally coherent Grothendieck category (see, for example, \cite[Theorems 10.1.3 and 16.1.14]{Pre09}, \cite[\S 1.2]{Kra00}, and \cite[Lemma 1.7]{Kra02}). It follows that the isomorphism classes of indecomposable injective objects of $((\cT^\cp)^\op, \Ab)$ form a small set (\cite[\S 3]{Kra97}). Moreover, by \cite[Lemma 1.7]{Kra00}, every injective object $E$ in $((\cT^\cp)^\op, \Ab)$ uniquely determines an object $X \in \cT$ (up to isomorphism) such that $E \cong \Hom_\cT(-, X)|_{\cT^\cp}$.  

The following theorem characterises the objects of $\cT$ that correspond to the injective objects of $((\cT^\cp)^\op, \Ab)$ in this way.  A triangle $X \to Y \to Z \to \Sigma X$ in $\cT$ is called \textit{pure} if the sequence \[ 0 \to \Hom_\cT(C, X) \to \Hom_\cT(C, Y) \to \Hom_\cT(C, Z) \to 0\] is exact in $\Ab$ for every $C\in \cT^\cp$.  An object $X \in \cT$ is called \textit{pure-injective} if every pure triangle of the form $X \to Y \to Z \to \Sigma X$ is a split triangle.  Recall that there is a unique morphism $X^{(I)} \to X$ induced by the identity $X_i := X \to X$ for each $i\in I$; we refer to this as the \textit{summation morphism}.

\begin{theorem}[{\cite[Theorem 1.8]{Kra00}}]\label{injectives}  The following statements are equivalent for $X \in \cT$. \begin{enumerate}
\item $X$ is a pure-injective object of $\cT$.
\item $\Hom_\cT(-, X)|_{\cT^\cp}$ is an injective object of $((\cT^\cp)^\op, \Ab)$.
\item For every small set $I$, the summation morphism $X^{(I)} \to X$ factors through the canonical morphism $X^{(I)} \to X^I$.
\end{enumerate}
\end{theorem}

A consequence of this theorem is that the isomorphism classes of indecomposable pure-injective objects in $\cT$ form a small set, which is called the \textit{Ziegler spectrum of $\cT$}.  We denote this set by $\Zg_\cT$.

\subsection{The topology on $\Zg_\cT$}
\label{TopZiegler}
In concurrent papers, Krause \cite{Kra97} and Herzog \cite{Her97} defined the (Ziegler) spectrum of a locally coherent Grothendieck category $\mathcal{G}$.  The space is given by the set $\Spec\mathcal{G}$ of the isomorphism classes of indecomposable injective objects in $\mathcal{G}$ and the open subsets parametrised by the Serre subcategories of the full subcategory $\fp(\mathcal{G})$ of finitely presented objects in $\mathcal{G}$; indeed, a typical open set is one of the form \[\{ E \in \Spec \mathcal{G} \mid \Hom_{\mathcal{G}}(S, E) \neq 0~\text{for  some}~S \in \mathcal{S}\}\] where $\mathcal{S}$ is a Serre subcategory of $\fp(\mathcal{G})$. Moreover, there is a basis of quasi-compact open subsets given by those of the form \[ \{ E \in \mathrm{Spec}(\mathcal{G}) \mid \Hom_\mathcal{G}(S, E) \neq 0\}\] where $S$ is an object of $\fp(\mathcal{G})$ (see \cite[Corollary 4.6]{Kra97} or \cite[Corollary 3.5]{Her97}).  

If we take $\mathcal{G} = ((\cT^\cp)^\op, \Ab)$, then it follows from \cref{injectives} (along with \cite[Corollary 1.9]{Kra00}) that the set $\Zg_\cT$ carries a topology, whose basis of open subsets correspond to the Serre subcategories of $\fp((\cT^\cp)^\op, \Ab)$.  

Following \cite{Kra02}, we will view the topology on $\Zg_\cT$ in terms of Serre subcategories of a more convenient category of covariant functors.  An additive functor $F \colon \cT \to \Ab$ is called \textit{coherent} if there exist a morphism $f \colon X \to Y$ in $\cT^\cp$ and an exact sequence
\[\Hom_\cT(Y, -) \xrightarrow{\Hom_\cT(f, -)} \Hom_\cT(X, -) \to F \to 0.\] 
We denote the category of coherent functors and natural transformations by $\mathrm{Coh}(\cT)$.  By \cite[Lemma 7.2]{Kra02}, there is an equivalence of categories between $\fp((\cT^\cp)^\op, \Ab)^\op$ and $\mathrm{Coh}(\cT)$ induced by the assignment taking a functor $S \colon (\cT^\cp)^\op \to \Ab$ to the coherent functor $F \colon \cT \to \Ab$ where $F(Z) := \Hom(S, \Hom_\cT(-, Z))$ for every $Z \in \cT$.  This equivalence yields the following characterisation of the topology on $\Zg_\cT$.

\begin{theorem}[{\cite[Fundamental correspondences]{Kra02}}]\label{Fundam}  There is an order-preserving bijection 
\[ \{\text{open sets of } \Zg_\cT\} \xrightarrow{\sim} \{\text{Serre subcategories of } \mathrm{Coh}(\cT)\} \]

\noindent given by $O \mapsto \mathcal{S}_O := \{F \in \mathrm{Coh}(\cT) \mid F(X) =  0~\text{for all}~X \in \Zg_\cT \setminus O\}$. 
Moreover, the sets \[ (F) := \{X \in \Zg_\cT \mid F(X) \neq 0 \} \] form a basis of quasi-compact open sets.
\end{theorem}

\subsection{Auslander--Reiten triangles}\label{AR-triangles}
Next we outline the connection between Auslander-Reiten triangles in $\cT^\cp$ and the isolated points of $\Zg_\cT$, i.e., those $X\in \Zg_\cT$ such that $\{X\}$ is an open subset.

A triangle \[ L \xrightarrow{f} M \xrightarrow{g} N \xrightarrow{h} \Sigma L\] in a triangulated category $\cC$ is called an \emph{Auslander--Reiten triangle}, or an \emph{AR-triangle (starting from $L$)}, if $f$ is left almost split and $N$ has local endomorphism ring.  Equivalently, we could require that $g$ is right almost split and $L$ has local endomorphism ring or, alternatively, that $f$ is left almost split and $g$ is right almost split (see \cite[Lemma 3.2]{Bel04}).

\begin{lemma}\label{AR isolated}
Suppose $L \in \Zg_\cT$ is compact as an object of $\cT$ and that 
there is an AR-triangle
\[  L \xrightarrow{f} M \xrightarrow{g} N \xrightarrow{h} \Sigma L\]
in $\cT^\cp$. 
Then $L$ is isolated in $\Zg_\cT$.
\end{lemma}
\begin{proof}
In \cite[Proposition 3.2]{Kra05a} it is shown that the triangle is also an AR-triangle in $\cT$.  We define a coherent functor $G_L$ by the following exact sequence 
\[\Hom_\cT(M, -) \xrightarrow{\Hom_\cT(f, -)} \Hom_\cT(L, -) \to G_L \to 0.\]  It follows from the definitions that, given an indecomposable object $X\in \cT^\cp$, we have $G_{L}(X)\neq 0$ if and only if $X\cong L$.  In other words $(G_L) = \{L\}$.
\end{proof}

\subsection{The existence of non-compact indecomposable pure-injective objects.}

We are now ready to connect properties of the Ziegler spectrum with the existence of non-compact indecomposable pure-injective objects.

\begin{proposition}\label{Existence}
Let $F \colon \cT \to \Ab$ be a coherent functor.  Assume the following conditions hold:
\begin{enumerate}[label=(\arabic*), font=\normalfont]
\item \label{AR}  Every $L\in (F)$ with $L$ compact as an object of $\cT$ admits an AR-triangle in $\cT^{\cp}$ starting from $L$.
\item \label{infinite} The open set $(F)$ contains, up to isomorphism, infinitely many objects in $\cT^\cp$. 
\end{enumerate}
Then there exists a non-compact indecomposable pure-injective object in $\cT$.
\end{proposition}

\begin{proof}
Suppose that $(F)$ consists, up to isomorphism, only objects in $\cT^\cp$.
By assumption, the open set $(F)$ is non-empty. Moreover every $L\in (F)$ admits an AR-triangle $L\xrightarrow{f} M\to N\to \Sigma L$ in $\cT^\cp$.  Then every $L\in (F)$ is isolated by \cref{AR isolated}.
Hence the open set $(F)$ can be written as the infinite (disjoint) union $\bigcup_{L\in (F)} \{L\}$ of open sets $\{L\}$, but this is impossible since $(F)$ is quasi-compact by \cref{Fundam}. Therefore $(F)$ must contain a point $M\in (F)$ which is not compact as an object in $\cT$.
\end{proof}

%%%%%%%%%%%%%%
%%%%%%%%%%%%%%
\newcommand{\etalchar}[1]{$^{#1}$}
\providecommand{\bysame}{\leavevmode\hbox to3em{\hrulefill}\thinspace}
\providecommand{\MR}{\relax\ifhmode\unskip\space\fi MR }
\providecommand{\MRhref}[2]{%
  \href{http://www.ams.org/mathscinet-getitem?mr=#1}{#2}
}
\providecommand{\href}[2]{#2}

\end{document}